\documentclass[a4paper, 10pt]{article}
\usepackage{geometry}
\usepackage[applemac]{inputenc}
\usepackage[T1]{fontenc}
\usepackage{amsmath,amsfonts, amssymb,enumerate,amsthm,amscd,graphicx,color,url}
\usepackage[english]{babel}
\usepackage{stmaryrd}
\usepackage{setspace}
\usepackage{algorithm,algorithmic}
\usepackage{multirow,multicol}
\usepackage[all]{xy}

\newcommand{\N}{\mathbb{N}}

\newcommand{\F}{\mathbb{F}}

\newcommand{\Hom}{\textrm{Hom}}

\newcommand{\MM}{M\!M}

\newtheorem{theo}{Theorem}

{\theoremstyle{plain}
\newtheorem{thm}{Theorem}[subsection]}
\newtheorem{lem}[thm]{Lemma}
\newtheorem{prop}[thm]{Proposition}
\newtheorem{cor}[thm]{Corollary}
{\theoremstyle{definition}
\newtheorem{ex}[thm]{Example}
\newtheorem{rmq}[thm]{Remark}
\newtheorem{defi}[thm]{Definition}}
\newcommand{\gal}{\textrm{Gal}}

\begin{document}
\title{Semi-characteristic polynomials, $\varphi$-modules and skew polynomials}
\author{Jérémy Le Borgne}
\date{}
\maketitle
\begin{abstract}
We introduce the notion of semi-characteristic polynomial for a semi-linear map of a finite-dimensional vector space over a field of characteristic $p$. This polynomial has some properties in common with the classical characteristic polynomial of a linear map. We use this notion to study skew polynomials and linearized polynomials over a finite field, giving an algorithm to compute the splitting field of a linearized polynomial over a finite field and the Galois action on this field. We also give a way to compute the optimal bound of a skew polynomial. We then look at properties of the factorizations of skew polynomials, giving a map that computes several invariants of these factorizations. We also explain how to count the number of factorizations and how to find them all.
\end{abstract}
\tableofcontents
\noindent \rule{\linewidth}{.4pt}
\begin{spacing}{1.2}
The aim of this paper is to relate the theory of $\varphi$-modules over a finite field, which is semi-linear algebra, with the theory of skew polynomials, and give some applications to understand better the factorization of skew polynomials over finite fields. Let $K$ be a field of characteristic $p>0$ endowed with a Frobenius morphism $\sigma$, and let $D$ be a finite-dimensional vector space over $K$ endowed with a map $\varphi$ that is semi-linear with respect to $\sigma$: this structure is called a $\varphi$-module. If one wants to evaluate a polynomial with coefficients in $K$ at such a semi-linear map, the ring of polynomials considered should have a natural structure of skew polynomial ring (or twisted-polynomial ring, as discussed by Kedlaya in \cite{ked}), in order for the relation $PQ(\varphi) = P(\varphi)Q(\varphi)$ to be valid. The theory of $\varphi$-modules has been widely investigated in $p$-adic Hodge theory, often as a tool in the theory of $(\varphi,\Gamma)$-modules that Fontaine introduced in \cite{fon} for the study of $p$-adic representations of local fields. On the other hand, the theory of skew polynomials was founded by Ore, who gave the fundamental theorems about such polynomials. In the context of finite fields, this theory has been recently used (for example by Boucher and Ulmer in \cite{felix}) to build error-correcting codes. One of Ore's main theorems concerns factorizations of skew polynomials, and it says that in two given factorizations of a polynomial, the irreducible factors that appear do not depend on the factorization up to similarity (similarity is an equivalence relation on skew polynomials, that generalize the notion of being equal up to multiplicative constant in the case of commutative polynomials, see section 1.1 for more detail). One very important result concerning skew polynomial rings over finite fields is a polynomial-time (in the degree of the polynomial) factorization algorithm due to Giesbrecht in \cite{gie2}.\\

In order to relate these theories, we introduce in the first part of the article the notion of semi-characteristic polynomial for a semi-linear map over a vector space over a field $K$ of characteristic $p>0$. Indeed, to a skew polynomial is naturally associated a $\varphi$-module (whose matrix is the companion matrix of the polynomial). Conversely, to a $\varphi$-module over $K$ we associate a skew polynomial with coefficients in $K$, the semi-characteristic polynomial. This polynomial should somehow behave like the characteristic polynomial of a linear map (in particular, its degree is the same as the dimension of the underlying vector space), except that it depends on the choice of some element in the $\varphi$-module $D$. For the purpose of this article, our definition of the semi-characteristic polynomial is mostly interesting in the case that the $\varphi$-module has a basis of the form $(x, \varphi(x), \dots, \varphi^{d-1}(x))$ for some $x \in D$. We give several properties of the semi-characteristic polynomial in this context, yielding the fact that the semi-characteristic polynomials given by two such $x$ are equivalent under the similarity relation, which is a crucial equivalence relation in the theory of skew polynomials. This polynomial is denoted by $\chi_{\varphi,x}$. We give an interpretation of a study of Jacobson about skew-polynomials in our context (see \cite{jac}, Chapter 1) to understand the factorizations of the semi-characteristic polynomial of a $\varphi$-module. Our Theorem \ref{JHsequences} gives a natural bijection between the Jordan-Hölder sequences of a $\varphi$-module and the factorizations of its semi-characteristic polynomial, the irreducible factors being given by the semi-characteristic polynomials of the composition factors. Conversely, our Proposition \ref{diviseursetquotients} shows any factorization of the semi-characteristic polynomial $\chi_{\varphi,x}$ yields a Jordan-Hölder sequence for the $\varphi$-module.
 
In the second part of the article, we use the preceding tools to study skew polynomials over finite fields. To a skew polynomial is naturally associated a so-called linearized polynomial, which is a polynomial of the form $\sum a_i X^{q^i}$, where the cardinal of the base field is a power of $q$. Linearized polynomials have long been related to skew polynomials (see for example \cite{ln}), and it is easy to see that the set of the roots of a linearized polynomial is a $\F_q$-vector space. On the other hand, to a $\varphi$-module is associated a linear representation of a Galois group by Fontaine's theory of $\varphi$-modules in characteristic $p$, which is recalled briefly in section 2.1. It appears that the considered representation is naturally the vector space of the roots of the associated linearized polynomial. As an application, we explain how to find the splitting field of a linearized polynomial together with the action of the Galois group on its roots:
\begin{theo}[Theorem \ref{splittingfield}]
Let $P \in \F_{q^r}[X,\sigma]$ with nonzero constant coefficient, and let $L_P$ be the associate linearized polynomial. Let $\Gamma$ be the companion matrix of $P$, and $\Gamma_0 = \Gamma \sigma(\Gamma) \cdots \sigma^{r-1}(\Gamma)$. Then the characteristic polynomial $Q$ of $\Gamma_0$ has coefficients in $\F_q$, and the splitting field of $L_P$ has dimension $m$ over $\F_{q^r}$, where $m$ is the maximal order of a root of $Q$ in $\overline{\F}_q$. Moreover, the action of a generator $g$ of the Galois group $G_{\F_{q^r}}$ is given in some basis of the $\F_q$-vector space of the roots of $L_P$ by the Frobenius normal form of $\Gamma_0$.
\end{theo}

We give a fast algorithm to compute a multiple of a skew polynomial that lies in the center of the ring, which has been a natural question about skew polynomials. In particular, since Giesbrecht's algorithm for factoring uses the computation of such a multiple, we improve the complexity of this part of his algorithm. We also explain how to test the similarity of skew polynomials effectively. Then, we investigate further the relations between the factorizations of a skew polynomial $P$ and the structure of the $\varphi$-module associated to $P$ as suggested by our Proposition \ref{submodule}. We define a map $\Psi$ that is shown to be multiplicative and to send a skew polynomial $P$ to a commutative polynomial of the same degree. This map allows to compute some invariants for $P$ such as the number and degrees of factors of $P$ in a given class of similarity, and tests irreducibility effectively. We explain how the factorization of the associated commutative polynomial $\Psi(P)$ yields one factorization of $P$ (and in fact, all of them) when $\Psi(P)$ is squarefree. At the level of $\varphi$-modules, this map classifies the $\varphi$-modules up to semi-simplification. We also show
\begin{theo}[Corollary \ref{anyorder}]
Let $P \in \F_{q^r}[X,\sigma]$. Then the similarity classes of irreducible factors of $P$ appear in all possible orders in the factorizations of $P$.
\end{theo}
We also use the map $\Psi$ to give a new way to compute the number of monic irreducible skew polynomials of given degree over a finite field. We then give a polynomial-time algorithm in the degree to count the number of factorizations of a skew polynomial as a product of monic irreducible polynomials, and explain a method to find them all (naturally, an algorithm for this would be exponential in general because so is the number of factorizations, however our method is linear in this number).\\

Note that $\varphi$-modules are often used as a tool for the study of Galois representations over local fields (and usually with characteristic zero). A $\varphi$-module over a local field has a sequence of \emph{slopes}, which are rational numbers that characterise the composition factors of the $\varphi$-module. It should be possible to recover the slopes of a $\varphi$-module over a field of positive characteristic from a factorization of its semi-characteristic polynomial, and even probably without factoring it using Newton polygons. This is not the approach of this paper, where the focus is on finite fields, but this topic will be discussed in the forthcoming paper \cite{leb}.\\

\section{The semi-characteristic polynomial of a $\varphi$-module}
Let $K$ be a field of characteristic $p$, let $a \geq 1$ be an integer, and $\sigma$ : $K \rightarrow K$ be the $a$-th power of the absolute Frobenius. Set $q = p^a$. The fixed field of $\sigma$ is the intersection of K with the finite field with $q$ elements $\F_q$. From now on, we assume that $\F_q \subset K$. If $P, Q \in K[X]$ and $\varphi$ is a semi-linear map on $K^d$, then it is not true in general that $PQ(\varphi) = P(\varphi) Q(\varphi)$, since $\varphi$ does not act trivially on $K$. The right point of view for the polynomials of semi-linear maps is that of \emph{skew polynomials}. Before investigating the properties of the semi-characteristic polynomials, we will recall some definitions and basic properties of the skew polynomial ring with coefficients in $K$.
\subsection{Skew polynomials}
\begin{defi}
The ring of skew polynomials with coefficients in $K$, denoted $K[X,\sigma]$, is the set of polynomials with coefficients in $K$ endowed with the usual addition, and the non-commutative multiplication $\cdot$ verifying $X\cdot a = \sigma(a)\cdot X$.
\end{defi}
\begin{defi}
Let $P,Q \in K[X,\sigma]$, we say that $P$ is a right-divisor of $Q$ (or that $P$ divides $Q$ on the right) if there exists $U \in K[X,\sigma]$ such that $Q = UP$. If this is the case, we say that $Q$ is a left-multiple of $P$. 
\end{defi}
The ring of skew polynomials was first studied by Ore in \cite{ore}. In this paper, he proves that the ring $K[X,\sigma]$ is a right-euclidean domain, and therefore a left-principal ideal domain. The notions of \emph{right greatest common divisor} (rgcd) and \emph{left lowest common multiple} (llcm) are well defined: we say that $D$ is a rgcd (resp. a llcm) of $P$ and $Q$ if $D$ is a left-multiple of any polynomial that divides both $P$ and $Q$ on the right (resp. if it is a right-divisor of any polynomial that is a left-multiple of both $P$ and $Q$). The same notions exist on the other side if $K$ is perfect. A factorization of a skew polynomial is in general not unique up to permutation of the factors and multiplication by a constant. Two different factorizations are related by the notion of similar polynomials, which we define now.
\begin{defi}
Two skew polynomials $P$ and $Q$ are said to be \emph{similar} if there exists $U \in K[X,\sigma]$, such that the right-greatest common divisor of $U$ and $P$ is 1, and such that $QU$ is the left-lowest common multiple of $U$ and $P$.
\end{defi}
\begin{thm}[Ore, \cite{ore}]
Let $P_1\cdots P_r = Q_1\cdots Q_s \in K[X,\sigma]$ be two factorizations of a given polynomial as a product of irreducible polynomials. Then $r=s$ and there exists a permutation $\sigma \in \mathfrak{S}_r$ such that $Q_{\sigma(i)}$ and $P_i$ are similar for all $1 \leq i \leq r$.
\end{thm}
We will use the notions of skew polynomials in a context of semilinear algebra, because these polynomials are naturally the polynomials of semilinear endomorphisms.
\subsection{Definition of the semi-characteristic polynomial}
As before, let $K$ be a field of characteristic $p$ and $\sigma$ : $K \rightarrow K$ be the $a$-th power of the absolute Frobenius. We still assume that $\F_q \subset K$. A $\varphi$-module over $K$ is a finite dimensional vector space $D$ endowed with a map $\varphi$ : $D \rightarrow D$ that is semi-linear with respect to $\sigma$. Such a $\varphi$-module is said to be \emph{étale} if the image of $\varphi$ contains a basis of $D$. The aim of this section is to associate to a $\varphi$-module a polynomial (or, more precisely, a family of polynomials) that is an analog of the characteristic polynomial for linear maps. In general, for $x \in D$, the set $I_{\varphi, x} = \{ Q \in K[X,\sigma] ~|~ Q(\varphi)(x) = 0\}$ is a left-ideal. Indeed, $I_{\varphi,x}$ is an additive subgroup of $K[X,\sigma]$, and if $P \in I_{\varphi,x}$ and $Q \in K[X,\sigma]$, then $QP(\varphi)(x) = Q(\varphi)(P(\varphi)(x)) = 0$. Hence this ideal has a generator $m_{\varphi,x}$ that we may call the minimal polynomial of $x$ under the action of $\varphi$.
We are mostly interested in the case where the degree of this polynomial is the dimension of the $\varphi$-module. In this case, we want to give an algebraic construction of $m_{\varphi,x}$ that will be called the semi-characteristic polynomial of $\varphi$ in $x$. The idea is that by Cramer's formulas, the coefficients of $m_{\varphi,x}$ are rational functions in the coefficients of the matrix of the map $\varphi$: indeed, if $(x, \varphi(x), \dots, \varphi^{d-1}(x))$ is a basis of the $\varphi$-module $D$ of dimension $d$, then $\varphi^d(x)$ can be written as a linear combination of  $x, \varphi(x), \dots, \varphi^{d-1}(x)$, the coefficients being of the form 
$$ \frac{\det(x, \varphi(x), \dots, \varphi^{i-1}(x), \varphi^d(x), \varphi^{i+1}, \dots, \varphi^{d-1}(x))}{\det(x, \varphi(x), \dots, \varphi^{d-1}(x))}. $$
We show that these coefficients are actually polynomials.\\

Let $d \in \N$ and let $A = K[(a_{ij})_{1\leq i,j \leq d}]$. Let
$$ G = \begin{pmatrix} a_{11} & \cdots & a_{1d}\\ \vdots & & \vdots\\ a_{d1} & \cdots & a_{dd}\end{pmatrix}$$
be the so-called generic matrix with coefficients in $A$.

\begin{thm}\label{existssemichar}
Let $x \in K^d \setminus \{0\}$, and let $\varphi$ be the $\sigma$-semi-linear map on $A^d$ whose matrix in the canonical basis is $G$. Then there exists a unique family of polynomials $P_0, \dots, P_{d-1} \in A$, depending only on $x$, such that
$$ \varphi^d(x) = P_{d-1} \varphi^{d-1}(x) + \cdots + P_1 \varphi(x) + P_0 x.$$
Moreover, each $P_i$ is an homogeneous polynomial in the coefficients of $G$.
\end{thm}
Before proving the theorem, let us mention the following corollary :
\begin{cor}\label{polynomial-x}
Let $L = K(x_1, \dots x_d)$, and $x = \begin{pmatrix} x_1\\ \vdots \\ x_d \end{pmatrix} \in L^d$. Let $P_i$ be the polynomials defined as in Theorem \ref{existssemichar}. Then the $P_i$ lie in $K[x_1, \dots, x_d][a_{ij}]$. In particular, we can define the $P_i$ for $x = 0$.
\end{cor}
The proof of the corollary will be given after that of the theorem. Let $x_0 \in K^d$, and let $(x_i)$ be the sequence of elements of $A^d$ defined by induction by $x_{i+1} = G \sigma(x_i)$. Here, $\sigma$ acts on a vector in $A^d$ by raising each coordinate to the same power $p^a$ as $\sigma$ on $K$. We call this sequence the sequence of iterates of $x_0$ under $G$. We will need the following lemma:
\begin{lem}\label{squarefree}
For all $x \in K^d \setminus \{0\}$, the determinant $\Delta = \det(x_0, \dots, x_{d-1})$ is an homogeneous element of $A$ that is squarefree.
\end{lem}
The fact that this determinant is homogeneous is clear since for all $i \geq 0$, the coefficients of $x_i$ are all homogeneous polynomials of degree $\sum_{j=0}^{i-1 }p^{j}$. We first show another lemma that will simplify the proof of Lemma \ref{squarefree}.
\begin{lem}\label{firstreduction}
With the above notations, it is enough to prove Lemma \ref{squarefree} for only one $x_0 \in K^d \setminus \{0\}$.
\end{lem}
\begin{proof}
It is harmless to assume that $K$ is algebraically closed (and hence infinite), which we will do in the proof. Let $x_0, x_0' \in K^d \setminus \{0\}$. Let $(x_i)$ (respectively $(x_i')$) the sequence of iterates of $x_0$ (respectively $x_0'$) under $G$. We assume that $\det(x_0, \dots, x_{d-1})$ is a squarefree polynomial. Let $P \in GL_d(K)$ such that $x_0' = Px_0$. Let $y_0 = x_0$ and let $(y_i)$ be the sequence of iterates of $y_0$ under $P^{-1}G\sigma(P)$. We have $x_1' = G\sigma(x_0') = G\sigma(P)\sigma(x_0) = Py_1$. An easy induction shows that for all $i \geq 0$, $x_i' = Py_i$. Hence, $\det(x_0', \dots, x_{d-1}') = \det P \det(y_0, \dots, y_{d-1})$. Since $\det P \in K^{\times}$, it is enough to show that $\det(y_0, \dots, y_{d-1})$ is squarefree. Define a morphism of $K$-algebras $\theta$ : $A \rightarrow A$ by $G \mapsto P^{-1}G\sigma(P)$ (this gives the image of all the indeterminates by $\theta$, and hence defines a unique morphism of $K$-algebras). In fact, this morphism is an isomorphism, with inverse given by $G \mapsto PG\sigma(P)^{-1}$.  The map $\theta$ extends naturally to $A^d$ and $A^{d \times d}$ and commutes with $\sigma$. By definition, $\theta(x_0') = x_0 = y_0$, and $\theta(x_{i+1}) = \theta(G)\sigma(\theta(x_i)) = P^{-1}G \sigma(P) \sigma(\theta(x_i))$, which shows by induction that for all $i \geq 0$, $\theta(x_i) = y_i$. Next, we remark that $ \theta(\det(x_0, \dots, x_{d-1})) =  \det(y_0, \dots, y_{d-1})$. Since $\theta$ is an isomorphism, is maps a squarefree polynomial to a squarefree polynomial (the quotient of $A$ by the ideal generated by a polynomial $Q$ is reduced if and only if $Q$ is squarefree). This proves the lemma.
\end{proof}
We can now prove Lemma \ref{squarefree}.
\begin{proof}
Let us prove the proposition by induction on the dimension $d$. Our induction hypothesis is that for any field $K$, the determinant $\Delta$ is a squarefree polynomial. If $d = 1$, then the result is obvious. Assume the proposition is proved for $d \in \N$, and prove it for $d +1$. Recall that any factorization of $\Delta$ has homogeneous irreducible factors. Therefore, we note that evaluating some of the variables to zero sends $\Delta$ to a squarefree polynomial (in the unevaluated variables) if and only if $\Delta$ is squarefree and the evaluation is nonzero. Indeed, if such an evaluation has a square factor, then so has $\Delta$. Conversely, if $\Delta$ has a square factor, then such an evaluation maps this square factor to either a nonconstant polynomial or to 0, and hence the evaluation also has a square factor.\\
We will evaluate some of the variables to zero, namely we look at
$$ G = \left( \begin{array}{cccc}
0 & 0 & \cdots & 0\\
X_0& \multicolumn{3}{c}{G'}
\end{array} \right),$$
where $X_0$ is of size $(d-1) \times 1$ and $G'$ is of size $(d-1)\times(d-1)$.\\
We define by induction $X_{i+1} = G'\sigma(X_i)$ for $i \geq 0$. Now let $x_0 = \begin{pmatrix} 1\\ 0 \\ \vdots \\ 0 \end{pmatrix}$, and $(x_i)$ the sequence of iterates of $x_0$ under this evaluation of $G$. Let us compute $(x_i)$. First, $x_1 = G\sigma(x_0) = \begin{pmatrix} 0\\X_0 \end{pmatrix}$ and $x_2 = \begin{pmatrix}  0 \\ X_1 \end{pmatrix}$. An easy induction shows that for all $i \geq 1$, $ x_i = \begin{pmatrix}  0\\ X_{i-1}\end{pmatrix}$.
Therefore, the evaluation $\Delta'$ of $\Delta$ that we are computing is
$$ \Delta' =\det(x_0, \dots, x_{d-1}) = \begin{vmatrix} 1 & 0 & 0 & \cdots & 0 \\
0 & X_0 & X_1 & \cdots & X_{d-2}\end{vmatrix}.$$
This determinant is equal to its lower right $(d-1)\times(d-1)$ minor, which is equal to $\det(X_0,\dots,X_{d-2})$. Denote by $S$ the set of variables appearing in $X_0$ (i.e., $a_{21}, \dots, a_{d1})$, and $S'$ the set of all the other variables appearing in $G'$. By induction hypothesis, applied with the field $K(S)$, the polynomial $\Delta' \in K(S)[S']$ is squarefree. This shows that if $\Delta'$ has a square factor, then the only variables appearing in this square factor lie in $S$. Hence it is enough to find an evaluation in $S'$ of $\Delta'$ that is squarefree to show that $\Delta'$ is squarefree, which implies that $\Delta$ is squarefree as well. We use the previous computation that we evaluate in $G = \left( \begin{array}{cccc}
0 & 0 & \cdots & 0\\
X_0& \multicolumn{3}{c}{I_{d-1}}
\end{array} \right),$ where $I_{d-1}$ is the $d-1$ identity matrix. Then $X_i = \sigma^{i}(X_0)$, and the evaluation of $\Delta'$ that we are computing is $\Delta'' = \det(X_0,\sigma(X_0), \dots, \sigma^{d-2}(X_0))$. What we need to prove now is that, in $K[b_1, \dots, b_d]$, the determinant
$$ V_q(b_1, \dots, b_d) = \begin{vmatrix} b_1 & b_1^q & \cdots & b_1^{q^{d-1}} \\ b_2 & b_2^q & \cdots & b_2^{q^{d-1}} \\ \vdots & \cdots & \cdots & \vdots \\ b_d & b_d^q & \cdots & b_d^{q^{d-1}} \end{vmatrix} $$
(which is known as the Moore determinant) is squarefree. It is a well-known fact that we have the following factorization:
$$ V_q(b_1, \dots, b_d) = c\prod_{\varepsilon \in \mathbb{P}^{d-1}(\F_q)} \sum_{i=1}^d \varepsilon_i b_i,$$
with $c \in \F_q^\times$. By $\varepsilon \in \mathbb{P}^{d-1}(\F_q)$, we mean that the considered $d$-uples $\varepsilon$ have their first nonzero coordinate equal to one. Note that if $\varepsilon = (\varepsilon_1, \cdots, \varepsilon_d) \in \mathbb{P}^{d-1}(\F_q)$, then $F_\varepsilon = \sum_{i=1}^d \varepsilon_i b_i$ is an irreducible polynomial (it is homogeneous with global degree 1). Two such distinct polynomials are not colinear since the coefficients are defined up to homothety, and hence are coprime. Moreover, if $F_\varepsilon(\beta_1, \cdots, \beta_d) =0$, then $\sigma$ being linear on $\F_q$ implies that the evaluation of the $q$-Vandermonde determinant at $(\beta_1,\cdots, \beta_d)$ is the determinant of a matrix whose rows are linearly dependant over $\F_q$, so this evaluation is zero. Hilbert's zeros theorem shows that $F_\varepsilon$ divides $V_q$, and given the coprimality of the $F_\varepsilon$'s, their product divides $V_q$. It is now enough to check that they have the same degree, that is easily seen to be $q^{d-1} + \cdots + q + 1 = \frac{q^d -1}{q-1}$.
\end{proof}
\begin{proof}[Proof of Theorem \ref{existssemichar}]
Since there exist simple $\varphi$-modules of dimension $d$ with coefficients in $A$ (that is, simple objects in the category of $\varphi$-modules over $A$, meaning that they have no nontrivial subspaces stable under the action of $\varphi$), the $\varphi$-module defined by $\varphi$ is simple as a $\varphi$-module over the field of fractions $B$ of $A$. In particular, for all $x \in K^d \setminus \{ 0\}$, there exists a unique family $F_0, \dots, F_{d-1} \in B$ such that $ \varphi^d(x) = F_{d-1} \varphi^{d-1}(x) + \cdots + F_1 \varphi(x) + F_0 x$. We want to show that the $F_i$'s are actually in $A$.\\
Let $x_0 = x$ and $(x_i)_{i \geq 0}$ be the sequence of iterates of $x$ under $G$: for $i \geq 0$, $x_i$ is the vector representing $\phi^i(x)$ in the canonical basis. According to Cramer's theorem, the $F_i$'s are given by the following formula, for $i \geq 0$,:
$$ F_i = \frac{\det(x_0, \dots, x_{i-1}, x_d, x_{i+1}, \dots, x_{d-1})}{\det(x_0, \dots, x_{d-1})}.$$
The denominator of $F_i$ is nothing but the determinant $\Delta$ from Lemma \ref{squarefree}. Since it is squarefree according to that lemma, Hilbert's zeros theorem (assuming $K$ is algebraically closed) shows that it is enough to prove that the numerator vanishes whenever the denominator vanishes. If $\Delta$ is mapped to zero by the evaluation of $G$ at $\underline{a}$, then the family $(x_0(\underline{a}), \dots, x_{d-1}(\underline{a}))$ is linearly dependent over $K$, so it spans a vector space of dimension at most $d-1$. But the span of this family is also stable under $\varphi$ since the smallest subspace stable by $\varphi$ containing $x_0(\underline{a})$ (that we denote $D_{x_0(\underline{a})}$) is spanned by the $x_i(\underline{a})$'s, and for all $r \in \N$, $x_{d+r}(\underline{a})$ lies in the span of $x_r(\underline{a}),\dots, x_{r+d-1}(\underline{a})$. Therefore, any family of $d$ elements of $D_{x_0(\underline{a})}$ is linearly dependent over $K$, so that the numerator of $F_i$ vanishes at $\underline{a}$ for all $0 \leq i \leq d-1$, which proves the theorem.
\end{proof}
\begin{proof}[Proof of Corollary \ref{polynomial-x}]
The case $d=1$ is obvious, so we will prove the corollary for $d \geq 2$. Recall that we want to show that, applying Theorem \ref{existssemichar} with the field $L = K(x_1, \dots, x_d)$, we get the fact that the $P_i$'s lie in $K(x_1, \dots, x_d)[a_{ij}]$. Now let $y = \begin{pmatrix} 1 \\0 \\ \vdots \\0 \end{pmatrix}$, $P = \begin{pmatrix} x_1 & 0 & \cdots & 0\\x_2 & 1 & \ddots & \vdots\\ \vdots & 0 & \ddots & 0\\ x_d &\cdots & 0 & 1 \end{pmatrix}$, and $H = P^{-1}G\sigma(P)$, so that $x = Py$, and the $P_i$ associated to $x$ (with respect to the matrix $G$) and $y$ (with respect to the matrix $H$) are the same (by the computation of Lemma \ref{firstreduction}). Since the coefficients of $H$ lie in $K[x_1, \dots, x_d][a_{ij}][x_1^{-1}]$, so do the $P_i$'s. Now taking another $P$ whose determinant is, say, $x_2$ (which is possible since $d \geq 2$), we see that the $P_i$'s also lie in $K[x_1, \dots, x_d][a_{ij}][x_2^{-1}]$, so they actually lie in $K[x_1, \dots, x_d][a_{ij}]$.
\end{proof}
\begin{defi}
With the previous notations, the polynomial $\chi^0_{\varphi,x} = X^{d} - \sum_{i=0}^{d-1} P_i X^i \in A[x_1, \dots, x_d][X, \sigma]$ obtained from the vector $x$ of Corollary \ref{polynomial-x} is called the \emph{universal semi-characteristic polynomial} in over $K$.\\
Given a $\varphi$-module $D$ of dimension $d$ over $K$, whose matrix in a basis $\mathcal{B}$ of $K^d$ is $G(\underline{a})$, and $x \in D$ whose coordinates in the basis $\mathcal{B}$ are given by $\underline{\xi} = (\xi_1, \dots, \xi_d)$, the \emph{semi-characteristic polynomial} of $\varphi$ in $x$ in the basis $\mathcal{B}$ is the evaluation of $\chi^{0}_{\varphi,x}$ at $\underline{a},\underline{\xi}$, that will be denoted $\chi_{\varphi,x,\mathcal{B}}$ or just $\chi_{\varphi,x}$ when no confusion is possible.
\end{defi}
\begin{rmq}
We will see later (see Corollary \ref{semichar-nongen}) that if the $\varphi$-module $D$ has dimension $d$, then $\chi_{\varphi,0,\mathcal{B}} = X^d$.
\end{rmq}
Of course, $\chi_{\varphi,x,\mathcal{B}}$ lies in $K[X,\sigma]$. The main properties of this polynomial are that $\chi_{\varphi,x,\mathcal{B}}(\varphi)(x) = 0$ and that $\chi_{\varphi,x,\mathcal{B}}$ is constructed algebraically from $G(\underline{a})$ and the coefficients of $x$. Let us give an example with $\sigma(x) = x^q$, $d = 2$, $G = \begin{pmatrix} a&b\\c&d \end{pmatrix}$, $x = \begin{pmatrix} 1\\0 \end{pmatrix}$. Then $\chi_{\varphi,x} = X^2 + P_1(a,b,c,d) X + P_0(a,b,c,d)$, with $P_0(a,b,c,d) = adc^{q-1}-bc^q$, $P_1(a,b,c,d) = -a^q - c^{q-1}d$. Note that, when formally putting $q = 1$, we recover the expression of the characteristic polynomial. When $d = 3$, with $x = \begin{pmatrix} 1\\0\\0 \end{pmatrix}$, we can write $\chi_{\varphi,x} = X^3 + P_2X^2+ P_1 X + P_0$. The polynomial $P_0$ has 336 terms, $P_1$ has 232 terms, and $P_2$ has 107 terms.
\begin{lem}
Let $(D,\varphi)$ be a $\varphi$-module of dimension $d$ over $K$, endowed with a basis $\mathcal{B}$. Let $x \in D$ and let $\chi_{\varphi,x}$ be the associated semi-characteristic polynomial. Then the left-ideal $I_{\varphi, x} = \{ Q \in K[X,\sigma] ~|~ Q(\varphi)(x) = 0\}$ contains $\chi_{\varphi,x}$, and if the $\varphi$-module $D$ is generated by $x$, then $I_{\varphi,x} = K[X,\sigma]\chi_{\varphi,x}$.
\end{lem}
\begin{proof}
We have already seen that $\chi_{\varphi,x}(\varphi)(x) = 0$. Assume $x$ generates $D$, then $x, \varphi(x), \dots, \varphi^{d-1}(x)$ is a linearly independant family in $D$, so all the nonzero elements of $I_{\varphi,x}$ have degree at least $d$. Then the monic generator of $I_{\varphi,x}$, which is the minimal polynomial $m_{\varphi,x}$ of $x$ by definition, has degree $d$. Since $I_{\varphi,x}$ contains a unique monic element of minimal degree, $m_{\varphi,x} = \chi_{\varphi,x}$. Hence $I_{\varphi,x} = K[X,\sigma]\chi_{\varphi,x}$.
\end{proof}
%\begin{lem}
%Let $D$ be a $\varphi$-module over $K$ and $x \in D$ a nonzero element. Then $x$ generates $D$ as a $\varphi$-module if and only if $\chi_{\varphi,x}$ is not divisible by $X$.
%\end{lem}
%\begin{proof}
%It is easy to see that, as for commutative polynomials, a skew polynomial is divisible by $X$ if and only if its constant coefficient is zero. Hence we want to prove that $x$ generates $D$ if and only if $\chi_{\varphi,x}$ has a nonzero constant coefficient. Let $\mathcal{B}$ be a basis of $D$ and let $x_0$ be the vector representing $x$ in this basis. Now let $d$ be the dimension of $D$, let $G$ be the generic matrix of size $d\times d$, and let $(x_i)_{i \geq 0}$ be the sequence of the iterates of $x_0$. Then:
%$$ \det G \cdot \det(x_0,\dots, x_{d-1})^q = \det(x_1,\dots, x_d),$$
%since the matrix whose columns are $(x_1, \dots, x_d)$ is built from the matrix whose columns are $(x_0, \dots, x_d)$ by raising all the coefficients to the $q$-th power, and multiplying by $G$ on the left. The linearity of the Frobenius map yields the result. Therefore, the constant coefficient in the generic semi-characteristic polynomial with first vector $x_0$ is $\pm \det G\cdot \det (x_0, \dots, x_{d-1})^{q-1}$. Evaluating $G$ to the matrix of $\varphi$ in the basis $\mathcal{B}$ shows that the constant coefficient of $\chi_{\varphi,x}$ is zero if and only if the family $x, \varphi(x), \dots, \varphi^{d-1}(x)$ is linearly dependant over $K$, which means exactly that $x$ does not generate the $\varphi$-module $D$.
%\end{proof}
The hypothesis that the $\varphi$-module admits a generator might not sound very satisfactory, but the following proposition shows that, at least when $K$ is infinite, such a generator always exists.
\begin{prop}
Assume the field $K$ is infinite, and let $(D,\varphi)$ be an étale $\varphi$-module over $K$. Then there exists $x \in D$ which generates $D$ under the action of $\varphi$.
\end{prop}
\begin{proof}
Let $d = \dim D$. What we want to prove is that there exists $x \in D$ such that the determinant $\det(x, \varphi(x), \dots, \varphi^{d-1}(x)) \neq 0$. We work with the field $K(X_1, \dots, X_d)$ over which the map $\sigma$ extends naturally, and we denote by $G$ the matrix of the map $\varphi$ in a basis of $D$. Let $x_0 = (X_1, \dots, X_d)$ and for $0 \leq i \leq d-1$, $x_{i+1} = G \sigma(x_i)$, we consider the polynomial $R = \det(x_0, \dots, x_{d-1})$. Since $K$ is infinite, it is enough to show that  $R$ is not the zero polynomial. Hence, it is enough to check that there is a specialization of $X_1, \dots, X_d$ in an algebraic closure $K^\text{alg}$ of $K$ such that $R$ is not sent to zero. Since $D$ is isomorphic over $K^\text{alg}$ to the $\varphi$-module whose matrix is identity, $R = \alpha V_q(X_1, \dots, X_d)$ where $\alpha \in K$ is nonzero and $V_q$ is the $q$-Vandermonde determinant that already appeared above, which is nonzero. This shows that $R \neq 0$.
\end{proof}
\begin{rmq}
We shall see later what happens when $K$ is finite. The matter of when a $\varphi$-module over $K$ has a generator is discussed in the section \ref{generators}. We will also see in the next section a more precise description of $\chi_{\varphi,x}$ when $x$ is not a generator (see Corollary \ref{submodule-0}).
\end{rmq}
\subsection{Basic properties}
We look at some of the first properties of the semi-characteristic polynomials that are directly related to Ore's theory of skew polynomials. In particular, we explain how the notion of similarity appears naturally in our context. Note that the results of this section and the next one are mostly a reinterpretation of Chapter I of \cite{jac}.
\begin{prop}\label{similarity}
Let $D$ be a $\varphi$-module over $K$ and $x,y \in D$ two nonzero elements. Assume that both $x$ and $y$ generate $D$ as a $\varphi$-module. Then $\chi_{\varphi,x}$ and $\chi_{\varphi,y}$ are similar.
\end{prop}
\begin{proof}
Since $x$ generates $D$, there exists $U \in K[X,\sigma]$ such that $y = U(\varphi)(x)$. Now, let $\mathcal{B}$ be a basis of $D$, and let $\chi_{\varphi,x},\chi_{\varphi,y}$ be the semi-characteristic polynomials corresponding to $x$ and $y$ in this basis. Then $\chi_{\varphi,y}U(\varphi)(x) = 0$, so  $\chi_{\varphi, x}$ is a right-divisor of $\chi_{\varphi,y}U$.\\
Since $y$ generates $D$, there exists a polynomial $V$ such that $x = V(\varphi)(y)$, so that $VU(\varphi)(x) = x$. Hence, $\chi_{\varphi, x}$ divides $VU-1$ on the right: this exactly means that the right greatest common divisor of $U$ and $\chi_{\varphi,x}$ is $1$. Hence, $\chi_{\varphi,y}U$, that has degree $\deg U + \deg \chi_{\varphi, x}$, is the left lowest common multiple of $U$ and $\chi_{\varphi,x}$. Conversely, let $P$ be a monic polynomial similar to $\chi_{\varphi,x}$. Then there exists $U \in K[X, \sigma]$ and $V \in K[X, \sigma]$ two polynomials, such that $U$ and $\chi_{\varphi,x}$ are right-coprime, and $V\chi_{\varphi,x} = PU$. Since $U$ and $\chi_{\varphi,x}$ are right-coprime, there exist $P_1, P_2 \in K[X, \sigma]$ such that $P_1 U + P_2 \chi_{\varphi,x} = 1$, so $P_1(\varphi) U(\varphi)(x) = x$. In particular, $y = U(\varphi)(x)$ is a generator of $D$. Moreover, $P(\varphi)(y) = 0$, so $\chi_{\varphi, y}$ is a right-divisor of $P$. Since they have the same degree and $P$ is monic, they are equal.
\end{proof}
\begin{cor}
Let $(D_1,\varphi_1)$ and $(D_2,\varphi_2)$ be two isomorphic $\varphi$-modules. Assume that $x_1 \in D_1$ generates $D_1$, and $x_2 \in D_2$ generates $D_2$. Then the semi-characteristic polynomials $\chi_{\varphi_1,x_1}$ and $\chi_{\varphi_2,x_2}$ are similar.
\end{cor}
\begin{proof}
Choosing an isomorphism mapping $x_1$ to some $x_2' \in D_2$, we are reduced to prove that $\chi_{\varphi_2,x_2}$ and $\chi_{\varphi_2,x_2'}$ are similar. This follows from Proposition \ref{similarity}.
\end{proof}
This shows that if $x$ generates the $\varphi$-module $D$, and $\mathcal{B}$ is a basis of $D$, then the semi-characteristic polynomial $\chi_{\varphi,x,\mathcal{B}}$ does not depend on the choice of the basis up to similarity. Frow now on, we shall sometimes talk about \emph{the} semi-characteristic polynomial of a $\varphi$-module, which will stand for the set of its characteristic polynomials with first vector generating the $\varphi$-module. It is contained in a similarity class that depends neither upon the class of isomorphism of the $\varphi$-module, nor upon the choice of the generator $x$. More precisely, it is the set of all monic polynomials contained in this similarity class. This set will be denoted by $\chi_\varphi$.\\
\subsection{Semi-characteristic polynomials and sub-$\varphi$-modules}
We now want to understand how the notion of semi-characteristic polynomial behaves with respect to sub-$\varphi$-modules.
\begin{prop}\label{irred-simple}
Let $(D,\varphi)$ be a $\varphi$-module of dimension $d$ over $K$. Then $\chi_{\varphi,x}$ is irreducible in $K[X,\sigma]$ for all $x \in D \setminus \{0\}$ if and only if $D$ is a simple $\varphi$-module.
\end{prop}
\begin{proof}
Assume that $\chi_{\varphi, x}$ is irreducible for all $x \in D \setminus \{0\}$. Using the above notations, for all $x$, $I_{\varphi,x} = K[X,\sigma]\chi_{\varphi,x}$. This means that for all nonzero polynomial $Q \in K[X,\sigma]$ with $\deg Q < d$, and for all $x$, $Q(\varphi)(x) \neq 0$. In particular, all the families $x, \dots, \varphi^{d-1}(x)$ are linearly independant over $K$, so $D$ is simple.\\
Conversely, assume that $D$ is simple. Let $x \in D$, nonzero. If $\chi_{\varphi,x} = PQ$ with $P,Q$ monic and $\deg P < d$, then $P(\varphi) Q(\varphi(x)) = 0$. The $\varphi$-module generated by $Q(\varphi(x))$ is then of dimension $<d$, so it is zero, and $Q(\varphi(x)) = 0$, which means that $\chi_{\varphi,x}$ divides $Q$ on the right, so $Q = \chi_{\varphi,x}$.
\end{proof}
\begin{prop} \label{submodule}
Let $0 \rightarrow D_1 \rightarrow D \rightarrow D_2 \rightarrow 0$ be an exact sequence of $\varphi$-modules, and denote by $\varphi_1, \varphi,\varphi_2$ the respective maps on $D_1, D, D_2$. Let $x \in D$. Let $\bar{x} = x$ mod $D_1$, and $x_1 = \chi_{\varphi_2, \bar{x}}(\varphi)(x)$. Then
$$ \chi_{\varphi, x} = \chi_{\varphi_1, x_1}\chi_{\varphi_2,\bar{x}}.$$
Moreover, if $x$ is a generator of $D$, then $x_1$ (resp. $\bar x$) is a generator of $D_1$ (resp. of $D_2$).
\end{prop}
\begin{proof}
We can assume that $K$ is algebraically closed. The set $\{x \in D ~/~ x \text{ generates } D\}$ is the Zariski-open subset of $D$ $\{\det(x, \varphi(x),\dots, \varphi^{d-1}(x))\neq 0  \}$. Since it is not empty, and the identity is polynomial in the coefficients of the vector $x$, it is enough to prove the proposition when $x$ is a generator of $D$. The result is clear if $D_1 = D$, so we assume that $D_1 \neq D$. In this case, $x \notin D_1$, otherwise $x$ would not generate the whole of $D$. Therefore, $\bar{x} \neq 0$. Let $x_1 = \chi_{\varphi_2, \bar{x}}(\varphi)(x)$. Since $\chi_{\varphi_2, \bar{x}}(\varphi)(\bar{x}) = 0$, $x_1 \in D_1$. It is a generator of $D_1$, otherwise there would be a polynomial $P$ with degree $< \dim D_1$ such that $P(\varphi) (x_1) = 0$, so $P\chi_{\varphi_2, \bar{x}}(\varphi)(x) = 0$, with $ \deg P\chi_{\varphi_2, \bar{x}} < \dim D$, which is in contradiction with the fact that $x$ generates $D$. On the other hand, it is obvious that $\bar{x}$ generates $D_2$. Hence, $\chi_{\varphi_1,x_1}\chi_{\varphi_2, \bar{x}}(\varphi)(x) = 0$, so $\chi_{\varphi_1,x_1}\chi_{\varphi_2, \bar{x}}$ is right-divisible by $P$, and the equality of the degrees proves that $ \chi_{\varphi, x} = \chi_{\varphi_1, x_1}\chi_{\varphi_2,\bar{x}}$.
\end{proof}
\begin{cor}\label{submodule-0}
Let $D$ be an étale $\varphi$-module over $K$ of dimension $d$ and $x \in D$. Assume that the sub-$\varphi$-module $D_x$ generated by $x$ has dimension $r \leq d$. Denote by $\varphi_x$ the map induced by $\varphi$ on $D_x$. Then $\chi_{\varphi,x} = \chi_{\varphi_x,x}X^{d-r}$.
\end{cor}
\begin{proof}
Apply Corollary \ref{submodule-0} with $D_1$ the sub-$\varphi$-module generated by $x$. Since $x$ is 0 in $D_2$ and $D_2$ has dimension $d- r$, it remains to show that $\chi_{\varphi_2,0} = X^{d-r}$. We will show by induction on $d$ that if $D$ is a $\varphi$-module of dimension $d$, then $\chi_{\varphi,0} = X^d$. Assume $K$ is algebraically closed. If $d = 1$, a direct computation shows that $\chi_{\varphi,0} = X$. For $d \geq 2$, $D$ is isomorphic to the $\varphi$-module whose matrix is identity since $K$ is algebraically closed, so we can pick a sub-$\varphi$-module $D_1$ of $D$ of dimension 1. By induction hypothesis and Proposition \ref{submodule}, we get the fact that $\chi_{\varphi, 0} = X\cdot X^{d-1} = X^d$.
\end{proof}
\begin{thm}\label{JHsequences}
Let $0 \subset D_m \subset \dots \subset D_0 = D$ be a Jordan-Hölder sequence for the $\varphi$-module $D$. Denote by $\varphi_i$ the map induced by $\varphi$ on $D_i$, and by $\overline{\varphi_i}$ the map induced by $\varphi$ on $D_i/D_{i+1}$ for $0 \leq i \leq m-1$. Let $x \in D$. Then
$$ \chi_{\varphi,x} = \pi_{m-1} \dots \pi_0,$$
with $\pi_i=\chi_{\overline{\varphi_i}, y_i}$ for some $y_i \in D_i/D_{i+1}$, for all $0 \leq i \leq m-1$ (in particular each polynomial $\pi_i$ is irreducible in $K[X, \sigma]$).
\end{thm}
\begin{proof}
Let us prove the theorem by induction on $m$. If $m=0$ the result is clear. Assume $m \geq 1$, then Proposition \ref{submodule} shows that  $\chi_{\varphi, x} = \chi_{\varphi_1, x_1}\chi_{\overline{\varphi_0},\bar{x}}$ with $\bar{x} = x$ mod $D_1$ and $x_1 = \chi_{\overline{\varphi_0}, \bar{x}}(x) \in D_1$. The induction hypothesis shows that $\chi_{\varphi_1, y_1}$ factors as $\prod_{i=m-1}^2 \chi_{\varphi_i, x_i}$ with $x_i \in D_i/D_{i+1}$.
\end{proof}
Let $D$ be an étale $\varphi$-module that has a generator $x$, and let $\chi_{\varphi,x}$ be its semi-characteristic polynomial with respect to $x$. If $D \rightarrow D_2$ is a quotient of $D$ (on which the induced map is denoted by $\varphi_2$), then we can associate to $D_2$ the semi-characteristic polynomial $\chi_{\varphi_2, \overline{x}}$ where $\overline{x}$ is the image of $x$ in the quotient. The following proposition shows that this map is in fact a bijection.
\begin{prop}\label{diviseursetquotients}
Let $D$ be an étale $\varphi$-module that has a generator $x$, and let $\chi_{\varphi,x}$ be its semi-characteristic polynomial with respect to $x$ in some basis. Then the above map is a natural bijection between the following sets:
\begin{enumerate}[(i)]
\item The monic right-divisors of $\chi_{\varphi,x}$;
\item The quotients of the $\varphi$-module $D$.
\end{enumerate}
Moreover, this bijection maps exactly the irreducible divisors to the simple quotients.
\end{prop}
\begin{proof}
First note that by Proposition \ref{submodule}, $\chi_{\varphi_2,\bar{x}}$ is an irreducible right-divisor of $P$. The considered map is surjective. Indeed, let $\chi_{\varphi,x} = P_1P_2$ with $P_2$ irreducible, and let $y = P_2(\varphi)(x)$. Let $D_y$ be the sub-$\varphi$-module generated by $y$ and let $D \rightarrow D/D_y$ be the canonical projection. Then $\bar{x} = x$ mod $D_y$ generates $D/D_y$ and $P_2(\varphi)(\bar{x}) = 0$, so $\chi_{\varphi,\bar{x}} = P_2$. Now, we show that the considered map is also injective. Let $D', D''$ be two simple quotients endowed with the induced maps $\varphi', \varphi''$. Denote by $\bar x'$ (resp. $\bar x''$) the image of $x$ by the canonical projection to $D'$ (resp. $D''$). Assume that $\chi_{\varphi',\bar{x}'} = \chi_{\varphi'',\bar{x}''}$. Then there exists a unique map $D' \rightarrow D''$ sending $\bar x '$ to $\bar{x}''$, and this map is an isomorphism. Moreover, the kernel of the composite map $D \rightarrow D''$ is the set of $y \in D$ such that $y = Q(\varphi)(x)$ for some polynomial $Q$ that is right-divisible by $\chi_{\varphi'',\bar{x}''}$. This is exactly the kernel of the canonical map $D \rightarrow D''$, so $D' = D''$. The fact that irreducible polynomials correspond to simple $\varphi$-modules follows from Proposition \ref{irred-simple}.
\end{proof}
We have the following corollary, that will be useful to count the factorizations of a skew polynomial over a finite field.
\begin{cor}
Let $P \in K[X,\sigma]$ be a monic polynomial with nonzero constant coefficient. Then there is a natural bijection between the factorizations of $P$ as a product of monic irreducible polynomials, and the Jordan-Hölder sequences of $V_P$.
\end{cor}
\begin{proof}
The proof is an easy induction on the number of irreducible factors of $P$.
\end{proof}
\section{Skew polynomials and $\varphi$-modules over finite fields}

We now want to focus on $\varphi$-modules over finite fields. We investigate some other links with skew polynomials and the so-called linearized polynomials. The reference used for basics on linearized polynomials over finite fields is \cite{ln}, Chap 4, §4.

\subsection{Galois representations and $\varphi$-modules}
Let $K$ be a field of characteristic $p$, and $K^\text{sep}$ be a separable closure of $K$. There is an anti-equivalence of categories between the $\F_q$-representations of $G_K = \gal(K^\text{sep}/K)$ and the étale $\varphi$-modules over $K$, when $\varphi$ acts as $x \mapsto x^q$ on $K$. The functor from representations to $\varphi$-modules is $\textrm{Hom}_{G_K}(\cdot,K^\text{sep})$, and its quasi-inverse is $\textrm{Hom}_\varphi(\cdot,K^\text{sep})$, where $G_K$ acts naturally on $K^\text{sep}$ and $\varphi$ acts as $x\mapsto x^q$ on  $K^\text{sep}$. This theory was introduced by Fontaine to study the Galois representations of local fields of characteristic $p$ and then gave birth to the theory of $(\varphi, \Gamma)$-modules to study the $p$-adic representations of $p$-adic fields. We want to use this tool in the context of finite fields. Here, we use the equivalence of categories to skew polynomials rather than representations. Indeed, the data of a representation of $G_{\F_{q^r}}$ is very simple: it is just given by the action of the Frobenius $x \mapsto x^{q^r}$ on the representation.\\
Let $p$ be a prime number, $q = p^a$ a power of $p$, and $r\geq 1$ be an integer.
\begin{defi}
A $q$-linearized polynomial over $\F_{q^r}$ is a polynomial $L \in \F_{q^r}[X]$ of the form $L = \sum_{i=0}^d a_i X^{q^i}$ where for all $0 \leq i \leq d$, $a_i \in \F_{q^r}$.
\end{defi}
\begin{rmq}
Such polynomials define $\F_q$-linear maps on any extension of $\F_{q^r}$, hence the terminology. Furthermore, the roots of a $q$-linearized polynomial have a natural structure of $\F_q$-vector space.
\end{rmq}
The vector space of $q$-linearized polynomials is endowed with a structure of noncommutative $\F_{q^r}$-algebra, with the usual sum and product given by the composition: $L_1\times L_2 (X) = L_1(L_2(X))$. It is easily checked that there is a natural isomorphism between the $\F_{q^r}$-algebras of $q$-linearized polynomials over $\F_{q^r}$, and of skew polynomials $\F_{q^r}[X,\sigma]$ where $\sigma(x) = x^q$. The correspondance between linearized and skew polynomials is the following:
\begin{defi}
Let $P \in \F_{q^r}[X, \sigma]$ be a skew polynomial, $P = \sum_{i=0}^d a_i X^i$. By definition, the associated linearized polynomial is $L_P = \sum_{i=0}^{d} a_i X^{q^i}$. Conversely, if $L$ is a linearized polynomial over $\F_{q^r}$, its associated skew polynomial is denoted by $P_L$.
\end{defi}
Let $P \in \F_{q^r}[X,\sigma]$ be a monic polynomial, say $P = X^d - \sum_{i=0}^{d-1} a_i X^i$. To $P$ we associate the linearized polynomial $L_P$ as before, and the $\varphi$-module $D_P$ given by the following data:
\begin{itemize}
\item $D_P = \bigoplus_{i=0}^{d-1} \F_{q^r} e_i$,
\item for $i \in \{0 ,\dots, d-2\}$, $\varphi(e_i) = e_{i+1}$,
\item $\varphi(e_{d-1}) = \sum_{i=0}^{d-1} a_i e_i$.
\end{itemize}
We can note that in the canonical basis $(e_0, \dots, e_{d-1})$, $\chi_{\varphi, e_0} = P$.
\begin{lem}\label{isomrep}
Let $(D,\varphi)$ be a $\varphi$-module over $\F_{q^r}$. Then $(D,\varphi^r)$ is a $\varphi^r$-module over $\F_{q^r}$. Let $V$ be the $\F_q$-representation of $G_{\F_{q^r}}$ associated to $\varphi$ and $V_r$ be the $\F_{q^r}$-representation of $G_{\F_{q^r}}$ associated to $\varphi^r$. Then $V_r \simeq V \otimes_{\F_q} \F_{q^r}$.
\end{lem}
\begin{proof}
It is clear that $\varphi^r$ is $\sigma^r$-semi-linear (which means just linear!). Then, there is a natural injective map $V = \textrm{Hom}_\varphi(D,\overline{\F}_{q}) \hookrightarrow \textrm{Hom}_{\varphi^r}(D,\overline{\F}_{q}) = V_r$. A classical argument (appearing for example in the proof of Proposition 1.2.6 in \cite{fon}) shows that a family of elements of $V$ that is linearly independant over $\F_q$ remains linearly independant over $\F_{q^r}$. Hence $V \otimes_{\F_q} \F_{q^r}$ injects into $V_r$, and since $\dim_{\F_q} V = \dim_{\F_{q^r}}V_r$, we get $V_r \simeq V \otimes_{\F_q} \F_{q^r}$.
\end{proof}
\begin{lem} \label{roots}
The representation $V_P$ associated to the $\varphi$-module $D_P$ is naturally isomorphic to the vector space of the roots of $L_P$ endowed with the natural action of $G_{\F_{q^r}}$.
\end{lem}
\begin{proof}
By definition, $V_P =  \textrm{Hom}_\varphi(D_P,\bar{\F}_q)$. Let $f \in V_P$. Let $\xi_i = f(e_i)$, the relations between the $e_i$'s and the fact that $f$ is $\varphi$-equivariant imply that for all $0 \leq i \leq d-2$, $\xi_i = \xi_0^{q^i}$. Subsequently, $\xi_0^{q^d}= \sum_{i=0}^{d-1} a_i \xi_0^{q^i}$, so that $\xi_0$ is a root of $L_P$. Conversely, the same computation shows that given any root $\xi$ of $L_P$, the map $D \rightarrow \overline{\F}_{q}$ sending $e_i$ to $\xi^{q^i}$ is a $\varphi$-equivariant morphism. Hence, the map $\left(\begin{array}{ccc} V_P & \rightarrow& \{\textrm{Roots of } L_P \}\\ f &\mapsto &f(e_0)\end{array}\right)$ is an isomorphism. Moreover, the action of $G_{\F_{q^r}}$ on $V_P$ comes from that on $\overline{\F}_{q^r}$, so that it is just the raising to the $q^r$-th power, and it is compatible with the previous map, making it an isomorphism of representations.
\end{proof}
\begin{thm}\label{characteristic}
Let $(D_P,\varphi)$ be the $\varphi$-module associated to the polynomial $P$. Then the Frobenius map $g \in G_{\F_{q^r}}$ acting on $V_P$ (by $x \mapsto x^{q^r}$) and the map $\varphi^r$ on $D_P$ have matrices that are conjugate.
\end{thm}
\begin{proof}
Since the matrix of the Frobenius acting on $V_P$ is conjugate to the matrix of the Frobenius acting on the $\F_{q^r}$-representation associated to $(D_P, \varphi^r)$ by Lemma \ref{isomrep}, we will show that the matrix of the latter is conjugate to the matrix of $\varphi^r$. By the Chinese remainders theorem, it is enough to show the result when the characteristic polynomial of $\varphi^r$ is a power of an irreducible polynomial $Q$. By the elementary divisors theory, it is also enough to show the result when $\varphi^r$ is a cyclic endomorphism. Assume that the matrix of $\varphi^r$ in the basis $(e_0, \dots e_{d-1})$ is the companion matrix of a polynomial $Q^e$, with $Q$ irreducible. Then the map $\Hom_{\varphi^r}(D_{Q^e},\overline{\F}_{q^r}) \rightarrow \left\{ \textrm{Roots of } L_{Q^e(X^r)}  \right\}= V_{Q^e(X^r)}$, mapping $f$ to $f(e_0)$ is an isomorphism of $\F_{q^r}$-representations by Lemma \ref{roots}. Let $\xi$ be a nonzero root of $L_{Q^e(X^r)}$, and let $f \in \Hom_{\varphi^r}(D_{Q^e},\overline{\F}_{q^r})$ mapping $e_0$ to $\xi$. Let $g$ be the Frobenius map on $V_{Q^e(X^r)}$, and $\chi_g$ be its characteristic polynomial. Then $f(\chi_g(\varphi^r)(e_0)) = \chi_g(g)(\xi) = 0$ because $f$ is $\varphi^r$-equivariant. Hence for all $f \in \Hom_{\varphi^r}(D_{Q^e},\overline{\F}_{q^r})$, $f(\chi_g(\varphi^r)(e_0)) = 0$. The injectivity of the map $\Hom_{\varphi^r}(D_{Q^e},\overline{\F}_{q^r}) \rightarrow V_{Q^e(X^r)}$ implies that $\chi_g(\varphi^r) = 0$, so the minimal polynomial of $\varphi^r$ divides $\chi_g$. If $Q^\varepsilon (g) = 0$ for some $\varepsilon \leq e$, then $Q^\varepsilon(\varphi^r) (f)(x) = Q^\varepsilon(g) (f(x))$ for all $x \in D$, $f \in  \Hom_{\varphi^r}(D_{Q^e},\overline{\F}_{q^r})$, and so $Q^\varepsilon(\varphi^r) = 0$, which shows that $\varepsilon = e$. Therefore, the minimal polynomial of $g$ is $Q^e$, and the matrices of $g$ and $\varphi^r$ are conjugate.
\end{proof}
From now on and throughout the article, we will perform complexity computations using the usual notations $O$ and $\tilde{O}$ (we say that a complexity is $\tilde{O}(g(n))$ if it is $O(g(n) \log^k(n))$ for some integer $k$). For computations over finite fields, we will usually express the complexity as the number of operations needed in the base field $\F_q$. We will make the common assumption that the multiplication in $\F_{q^r}$ is quasilinear. We will denote by $\MM(d)$ the complexity of the multiplication of two matrices of size $d \times d$ over $\F_q$, so that the multiplication of two matrices of size $d \times d$ over $\F_{q^r}$ is $\MM(dr)$.
\begin{rmq}\label{complexfrob}
The matrix of $g$ can be computed in $O(\MM(dr) \log r + d^2 r^2 \log q\log r)$ multiplications in $\F_{q}$ if $P$ has degree $d$ and multiplication of matrices of size $d\times d$ over $\F_{q^r}$ has complexity $\MM(dr)$. Indeed, if $G$ is the companion matrix of $\varphi$, then the matrix of $\varphi^r$ is $G\sigma(G)\cdots\sigma^{r-1}(G)$. Since applying $\sigma^t$ to an element of $\F_{q^r}$ costs $t \log q$ operations in $\F_{q^r}$ (by a fast exponentiation algorithm), a divide-and-conquer algorithm allows us to compute the matrix of $\varphi^r$ with $O(\MM(d) \log r + d^2 r)$ operations in $\F_{q^r}$. %The reader should not be surprised by the complexity: the algorithm written below clearly shows that at each step of the recursion, a power of $\sigma$ is applied to a matrix, so even with fast exponentiation we cannot expect a better complexity in $r$.
\end{rmq}
\begin{algorithm}
\caption{Returns $E(G,r) = G\sigma(G)\cdots \sigma^{r-1}(G)$}
\begin{algorithmic}
\REQUIRE $G$ the matrix of $\varphi$, $r \geq 1$ an integer
\ENSURE $E(G,r)$ the matrix of $\varphi^r$
\IF{r = 1}
\RETURN $G$
\ELSE \IF{r is even}
\RETURN $E(G, r/2)\cdot\sigma^{r/2}(E(G,r/2))$
\ELSE 
\RETURN $G\cdot\sigma(E(G, (r-1)/2)\cdot\sigma^{(r-1)/2}(E(G,(r-1)/2)))$
\ENDIF
\ENDIF
\end{algorithmic}
\end{algorithm}
The following proposition gives another application of considering $\varphi^r$, namely testing similarity of polynomials.
\begin{prop}\label{testsimilarity}
Let $P,Q \in \F_{q^r}[X, \sigma]$ be two monic polynomials. Let $\Gamma_P$ (resp. $\Gamma_Q$) be the companion matrix of $P$ (resp. $Q$). Then $P$ and $Q$ are similar if and only if the matrices $\Gamma_P\cdots \sigma^{r-1}(\Gamma_P)$ and $\Gamma_Q\cdots \sigma^{r-1}(\Gamma_Q)$ are conjugate.
\end{prop}
\begin{proof}
Assume $P$ and $Q$ are similar, and let $d = \deg P = \deg Q$. Then $P = \chi_{\varphi,x}$ for some $x \in D_Q$. Therefore, there exists $U \in GL_d(\F_{q^r})$ such that $\Gamma_P = U^{-1}\Gamma_Q \sigma(U)$, so $\Gamma_P\cdots \sigma^{r-1}(\Gamma_P) = U^{-1}\Gamma_Q\cdots \sigma^{r-1}(\Gamma_Q) U$. Hence, these two matrices are conjugate.\\
Conversely, assume that these matrices are conjugate. Then the representations $V_{Q}$ and $V_P$ are isomorphic (because these matrices are conjugate to the matrices of the action of the Frobenius on the respective representations, which are therefore $\F_q$-conjugate), so the $\varphi$-modules $D_P$ and $D_Q$ are isomorphic. Hence, $P$ and $Q$ are similar.
\end{proof}
This proposition shows how similarity can be tested only by computing the Frobenius normal form of $\varphi^r$, which can be done in $\tilde{O}(\MM(dr))$ operations in $\F_q$.
\subsection{The splitting field of a linearized polynomial}
In this section, we use the previous results to explain how to compute the splitting field of a $q$-linearized polynomial with coefficients in $\F_{q^r}$ and the action of $G_{\F_{q^r}}$ on its roots. We start with a lemma from \cite{ln}.
\begin{lem} \label{lidl}
Let $Q \in \F_q[Y]$ with nonzero constant coefficient and let $L_Q$ be the $q$-linearized associated polynomial. Then the splitting field of $L_Q$ is $\F_{q^m}$, where $m$ is the maximal order of a root of $Q$ in $\overline{\F}_q^\times$.
\end{lem}
\begin{rmq}
If $Q = Q_1^{t_1} \cdots Q_s^{t_s}$ with the polynomials $Q_i$ distinct monic irreducible, then setting $t = \max\{t_i\} -1$ and $e = \max \{ \textrm{Order of the roots of } Q_i\}$, we have $m = ep ^t$.
\end{rmq}
We are ready to prove the following:
\begin{thm} \label{splittingfield}
Let $P \in \F_{q^r}[X,\sigma]$ with nonzero constant coefficient, and let $L_P$ be the associated linearized polynomial. Let $\Gamma$ be the companion matrix of $P$, and $\Gamma_0 = \Gamma \sigma(\Gamma) \cdots \sigma^{r-1}(\Gamma)$. Then the characteristic polynomial $Q$ of $\Gamma_0$ has coefficients in $\F_q$, and the splitting field of $L_P$ has degree $m$ over $\F_{q^r}$, where $m$ is the maximal order of a root of $Q$ in $\overline{\F}_q$.\\
Moreover, the Frobenius normal form $G_0$ of $\Gamma_0$ has coefficients in $\F_q$, and the action of a generator $g$ of the Galois group $G_{\F_{q^r}}$ is given (in some basis of the $\F_q$-vector space of the roots of $L_P$) by $G_0$. 
\end{thm}
\begin{proof}
By Theorem \ref{characteristic}, $\Gamma_0$, being the matrix of $(D_P, \varphi^r)$, is conjugate to the matrix of the action of $g$ on the roots of $L_P$. The latter has coefficients in $\F_q$ since it is the matrix of an $\F_q$-representation $V_P$ of $G_{\F_{q^r}}$, as does the Frobenius normal form of $\Gamma_0$. It only remains to determine the splitting field of $L_P$. But this field is the same as the subfield of $\overline{\F}_{q^r}$ fixed by the kernel of the representation $V_P$. This again is the same as the subfield fixed by the kernel of the representation $V_P \otimes \F_{q^r}$, which is the same as the splitting field of $L_Q$. The result then follows from Lemma \ref{lidl}.
\end{proof}
\begin{ex}
Let $q = 7$ and $r = 5$. The field $\F_{7^5}$ is built as $\F_7[Y]/(Y^5 + Y + 4)$, and $\omega$ denotes the class of $Y$ in $\F_{7^5}$. Let $L = Z^{7^3} + \omega Z^{7^2} - \omega^2 Z \in \F_{7^5}[Z]$. The associated skew polynomial $P \in \F_{7^5}[X,\sigma]$ is $X^3 +\omega X^2 -\omega^2$, so the matrix of $\varphi$ on $D_P$ is $\Gamma = \begin{pmatrix}  0 & 0 & \omega^2\\ 1 & 0 & 0\\ 0 & 1 & -\omega \end{pmatrix}$. The characteristic polynomial of the matrix $\Gamma_0$ of $\varphi^r$ is $Y^3 + Y^2 + Y +5$, which is irreducible. The order of any root of this polynomial in $\overline{\F}_7$ is 171, so the splitting field of $L$ is $\F_{7^{5 \times 171}}$. This also shows that the Jordan form of the matrix of $\varphi^r$ is $\begin{pmatrix}  0 & 0 & -5\\ 1 & 0 & -1 \\ 0 & 1 & -1\end{pmatrix}$, that is the matrix of the action of $x \mapsto x^{7^5}$ on a basis of the roots of $L$ over $\F_{7^5}$.
\end{ex}
\subsection{Optimal bound of a skew polynomial}
The previous section has shown that, given a $\varphi$-module $(D,\varphi)$ over $\F_{q^r}$, the $\varphi^r$-module $(D,\varphi^r)$ should have interesting properties for the study of $(D,\varphi)$. In this subsection, we will show how this idea can also help us solve the problem of finding a multiple of a polynomial lying in the center of $\F_{q^r}[X,\sigma]$. We recall the notations from the above section, that we shall use in this one: if $P\in \F_{q^r} [X, \sigma]$, then $D_P$ is the associated $\varphi$-module, $L_P$ is the associated linearized polynomial, and $V_P$ is the associated linear representation of $G_{\F_{q^r}}$ (either by Fontaine's theory, or as the roots of $L_P$, since both are the same object by Lemma \ref{roots}).

The following lemma is a slightly generalized version of a lemma from \cite{ln}, where only the case $r = 1$ is treated.
\begin{lem}\label{divclasslin}
Let $L_1, L$ be two linearized polynomials. Then $L_1$ is a right-divisor of $L$ in the algebra of linearized polynomials if and only if it is a divisor of $L$ in the classical sense.
\end{lem}
\begin{proof}
First assume that $L_1$ divides $L$ on the right in the sense of linearized polynomials, meaning that there exists a linearized polynomial $L_2$ such that $L(X) = L_2(L_1(X))$. Since the constant coefficient of $L_2$ is zero, this implies that $L_1$ divides $L$ in the classical sense.\\
Conversely, if $L_1$ divides $L$ in the classical sense, write the right-euclidean division of $L$ by $L_1$ as linearized polynomials, we have $L = L_2\circ L_1 + R$ with $\deg R < \deg L_1$. From the first part of the proof, $L_1$ divides $L_2 \circ L_1$ in the classical sense, so it also divides $R$. Since $\deg R < \deg L_1$, $R = 0$.
\end{proof}
This already allows us to give an explicit description of the \emph{optimal bound} of a skew polynomial. A bound of a skew polynomial $P$ is a nonzero multiple of $P$ that lies in the center of $\F_{q^r}[X,\sigma]$ (which is easily shown to be $\F_q[X^r]$), and an optimal bound is a bound with lowest degree.
\begin{thm}\label{borneopti}
Let $P \in \F_{q^r}[X,\sigma]$ be a monic polynomial with nonzero constant term. Let $\pi_{\varphi^r}$ be the minimal polynomial of the $\F_q$-linear map $\varphi^r$ : $D_P \rightarrow D_P$. Then the optimal bound for $P$ is $\pi_{\varphi^r}(X^r)$. It has degree at most $r\deg P$ and can be computed in $\tilde{O}(d^2 r^2\log q + MM(rd))$ operations in $\F_q$.
\end{thm}
\begin{proof}
Since $\pi_{\varphi^r}(\varphi^r) = 0$, $\pi_{\varphi^r}(X^r)(\varphi)(e_0) = 0$. Since $P = \chi_{\varphi, e_0}$, $P$ is a right divisor of $\pi_{\varphi^r}(X^r)$. Conversely, if $Q\in\F_q[Y]$ is a monic polynomial such that $Q(X^r)$ is right-divisible by $P$, then $Q(\varphi^r)(e_0) = 0$. Moreover, since $Q(X^r)$ is central, $XQ(\varphi^r)(e_0) = Q(\varphi^r)(\varphi(e_0))$. An immediate induction shows that, since $e_0$ generates $D_P$ under the action of $\varphi$, $Q(\varphi^r) = 0$. Hence, $\pi_{\varphi^r}$ divides $Q$.\\
By Remark \ref{complexfrob}, the matrix of $\varphi^r$ can be computed in $\tilde{O}(d^2 r^2\log q + \MM(rd))$ operations in $\F_q$. Its minimal polynomial can be computed in $\tilde{O}(\MM(rd))$ operations by \cite{gie}, hence the complexity of the computation of the optimal bound.
\end{proof}
\begin{rmq}
This complexity can be compared with Giesbrecht's computation of an optimal bound in $\tilde{O}(d^3r^2 + \MM(rd))$ operations in $\F_{q^r}$ (\cite{gie2}, Lemma 4.2). Since this part is used in his factorization algorithm, computing the optimal bound using Theorem \ref{borneopti} improves the complexity of this part of Giesbrecht's algorithm.
\end{rmq}
Theorem \ref{splittingfield} has shown that the characteristic polynomial of $\varphi^r$ already gives interesting information. Since the characteristic polynomial is also slightly easier to compute than the minimal polynomial, we introduce the following definition:
\begin{defi}
Let $P \in \F_{q^r}[X, \sigma]$ a monic skew polynomial, and let $(D_P,\varphi)$ be the associated $\varphi$-module. The polynomial $\Psi(P) \in \F_q[Y]$ is defined as the characteristic polynomial of $\varphi^r$, that is $\det(Yid - \varphi^r)$.
\end{defi}
\begin{rmq}
The polynomial $\Psi(P)$ can be thought of as lying in the center $\F_q[X^r]$ of $\F_{q^r}[X,\sigma]$. This is a reason why we use the variable $Y$. The other reason is that $\Psi(P)$ is a commutative polynomial, and a different notation for the variable can help avoid confusions. 
\end{rmq}
\begin{rmq}
By a result of Keller-Gehrig, the characteristic polynomial of an endomorphism of $\F_{q^r}^d$ can be computed in $O(\MM(dr))$ operations in $\F_{q}$ (see \cite{kg}). Hence, if $P$ has degree $d$, $\Psi(P)$ can be computed in $O(\MM(dr)\log r + d^2 r^2\log r\log q)$ operations in $\F_q$.
\end{rmq}

\begin{cor}
Let $P \in \F_{q^r}[X,\sigma]$ with nonzero constant coefficient. The polynomial $\Psi(P)(X^r)$ is a bound for $P$.
\end{cor}
\begin{proof}
It follows directly from Theorem \ref{borneopti}.
\end{proof}
\subsection{The map $\Psi$ and factorizations }
In  this section, we explain how the map $\Psi$ can be used to find factorizations of a skew polynomial.\\

Let $(D,\varphi)$ be an étale $\varphi$-module over $\F_{q^r}$.

\begin{prop}\label{psiconstant}
The map $\Psi$ is constant on similarity classes.
\end{prop}
\begin{proof}
Assume that the $\varphi$-module $D$ is generated by some $x \in D$, and let $\chi_{\varphi,x}$ be the corresponding semi-characteristic polynomial of $\varphi$ in the basis $x, \varphi(x), \dots, \varphi^{d-1}(x)$. Then it is clear from the definition that $\Psi (\chi_{\varphi,x})$ is the characteristic polynomial of $\varphi^r$. Now, if two polynomials $P$ and $Q$ are similar, $Q$ appears as the semi-characteristic polynomial of an element of $D_P$. Hence, there exists $x \in D_P$ such that $\chi_{\varphi,x} = Q$. In this case, $\Psi(Q)$ is the characteristic polynomial of $\varphi^r$, which also equals $\Psi(P)$, so $\Psi(P) = \Psi(Q)$.
\end{proof}
As we will see below, $\Psi$ does not classify the similarity classes of polynomials, but it classifies the similarity classes of the factors appearing in the factorizations of a polynomial.
\begin{prop}\label{multiplicative}
Let $P,Q \in \F_{q^r}[X,\sigma]$ be two monic, nonconstant polynomials. Then $\Psi(PQ) = \Psi(P)\Psi(Q)$.
\end{prop}
\begin{proof}
We translate Proposition \ref{submodule} in terms of matrices: let $(e_0, \dots, e_{d-1})$ be the canonical basis of $D_{PQ}$, and let $y = Q(\varphi)(e_0)$. Let $\delta_1 = \deg P$ and $\delta_2 = \deg Q$, then $(y, \varphi(y),\dots, \varphi^{\delta_1-1}(y), e_0, \dots, e_{\delta_2-1})$ is a basis of $D_{PQ}$ in which the matrix of $\varphi$ is $H = \left( \begin{array}{c|c} G_P &(\star)\\ \hline 0 &G_Q\end{array} \right)$, where $G_P$ (resp. $G_Q$) is the companion matrix of $P$ (resp. of $Q$). Since this matrix is block-upper-triangular, the characteristic polynomial of $H\sigma(H) \cdots \sigma^{r-1}(H)$ is $ \Psi(P)\Psi(Q)$. On the other hand, since $H$ is the matrix of $\varphi$ in some basis, this characteristic polynomial is $\Psi(PQ)$.
\end{proof}
\begin{prop}\label{similardivisor}
Let $P,Q \in \F_{q^r} [X,\sigma]$ be two monic polynomials with $P$ irreducible. Then $P$ is similar to a right-divisor of $Q$ if and only if $\Psi(P)$ divides $\Psi(Q)$. 
\end{prop}
\begin{proof}
The case $P = X$ is obvious, so we treat the case where $P$ has nonzero constant coefficient. If $P$ is a right-divisor of $Q$, then Proposition \ref{multiplicative} shows that $\Psi(P)$ divides $\Psi(Q)$. Conversely, if $\Psi(P)$ divides $\Psi(Q)$, we want to show that $D_P$ is a quotient of $D_Q$, or equivalently, that $V_P$ is a subrepresentation of $V_Q$. Let $g$ be the Frobenius map $x \mapsto x^{q^r}$ acting on $V_Q$. We want to show that $V_Q$ has a subspace stable under $g$, of dimension $\deg P$, on which the characteristic polynomial of $g$ is $\Psi(P)$. By the Chinese remainders Theorem, we can assume that $\Psi(Q)$ is a power of $\Psi(P)$. Indeed, this Theorem shows that $D_Q$ is the direct sum of subspaces stable under the action of $g$, and such that the characteristic polynomial of the action of $g$ on each of these subspaces is a power of an irreducible polynomial. Now, assuming that $\Psi(Q)$ is a power of $\Psi(P)$, we see that the Jordan form of $g$ on $V_Q$ is a block-upper-triangular matrix whose diagonal blocks are all the same, equal to the companion matrix of $\Psi(P)$. Thus $V_P$ appears as a subrepresentation of $V_Q$.
\end{proof}
\begin{cor}\label{anyorder}
Let $P \in \F_{q^r}[X,\sigma]$. Then the similarity classes of irreducible factors of $P$ appear in all possible orders in the factorizations of $P$.
\end{cor}
\begin{proof}
If the similarity class of $P_0 \in K[X,\sigma]$ appears in some factorization of $P$, then $\Psi(P_0)$ is a divisor of $\Psi(P)$, and $P$ has a right-divisor similar to $P_0$. For the general case, it is easy to see that if $Q \in \F_{q^r}[X,\sigma]$ is any polynomial, then there exists $\tilde Q$ similar to $Q$ such that $XQ = \tilde Q X$.
\end{proof}
\begin{cor} \label{psicaracterise}
Let $P,Q \in \F_{q^r}[X,\sigma]$ be two monic polynomials, with $P$ irreducible. Then $P$ and $Q$ are similar if and only if $\Psi(P) = \Psi(Q)$.
\end{cor}
\begin{proof}
Since we already know that $\Psi$ is constant on similarity classes, it is enough to prove that if $\Psi(P) = \Psi(Q)$, then $P$ and $Q$ are similar. If $\Psi(P) = \Psi(Q)$, then $Q$ has a right-divisor that is similar to $P$. Since $P$ and $Q$ have the same degree, $P$ is similar to $Q$.
\end{proof}
We note that this property is not true when $P$ is not irreducible: this will be discussed in Remark \ref{psietsemisimplification} below.
\begin{lem} \label{surjirred}
Every irreducible monic polynomial in $\F_q[Y]$ is the image of a monic irreducible polynomial by the map $\Psi$.
\end{lem}
\begin{proof}
It is clear that $Y$ is the image of $X$. Now assume that $R \in \F_q[Y]$ is irreducible, monic, with nonzero constant coefficient. Let $V$ be the $\F_q$-representation of $G_{\F_{q^r}}$ whose dimension is the degree of $R$, and for which the matrix of the Frobenius map $x \mapsto x^{q^r}$ is the companion matrix of $R$. Let $(D,\varphi)$ be the $\varphi$-module corresponding to $V$. Since $R$ is irreducible, $V$ is an irreducible representation, so $(D,\varphi)$ is an irreducible $\varphi$-module. Let $\chi_{\varphi}$ be the semi-characteristic polynomial of $\varphi$ at some nonzero $x \in D$. Theorem \ref{characteristic} shows that $R$ is equal to the characteristic polynomial of $\varphi^r$, which is just $\Psi(\chi_\varphi)$ by definition. Moreover, the irreducibility of the $\varphi$-module $D$ yields the irreducibility of $\chi_{\varphi,x}$.
\end{proof}
\begin{cor}
Let $P \in \F_{q^r}[X,\sigma]$ be a monic polynomial. The polynomial $P$ is irreducible if and only if $\Psi(P)$ is irreducible in $\F_q[Y]$.
\end{cor}
Since testing irreducibility of a polynomial of degree $d$ over $\F_q$ can be done in $O(d\MM(d))$ mulitplications in $\F_q$, we can test irreducibility of a polynomial in $\F_{q^r}[X,\sigma]$ of degree $d$ in $O(d^2r^2\log r\log q + \MM(rd) + d\MM(d))$ multiplications in $\F_{q}$.
\begin{proof}
We can assume that $P$ has nonzero constant coefficient. By Proposition \ref{multiplicative}, we know that if $\Psi(P)$ is irreducible, then so is $P$. Conversely, Lemma \ref{surjirred} shows that every irreducible divisor $D$ of $\Psi(P)$ has an irreducible antecedent $Q$ by $\Psi$. By Proposition \ref{similardivisor}, $Q$ is similar to a right-divisor of $P$. Hence since $P$ is irreducible, $Q = P$, and $\Psi(P) = D$, so $\Psi(P)$ is irreducible.
\end{proof}
\begin{cor}
The map $\Psi$ is surjective.
\end{cor}
\begin{proof}
The result follows directly from Lemma \ref{surjirred} and Proposition \ref{multiplicative}.
\end{proof}
\begin{cor}
The degrees of the factors in a factorization of a monic polynomial $P \in \F_{q^r}[X, \sigma]$ as a product of irreducibles are the same as the degrees of the factors of $\Psi(P)$ in a factorization as a product of irreducible polynomials in $\F_q[Y]$.
\end{cor}
\begin{rmq} \label{psietsemisimplification}
Let $P,Q \in \F_{q^r}[X, \sigma]$ be two monic polynomials, then $\Psi(P) = \Psi(Q)$ if and only if $D_P$ and $D_Q$ have the same semi-simplifications. Indeed, the similarity classes (with multiplicities) of the monic irreducible factors appearing in factorizations of $P$ are uniquely determined by $\Psi(P)$ because $\Psi$ is multiplicative and by Ore's Theorem. On the other hand, these similarity classes are also uniquely determined by the semi-simplification of $D_P$ again by multiplicativity of $\Psi$ and by Theorem \ref{JHsequences}.
\end{rmq}
Now, we use the map $\Psi$ to get more precise information about the factorization of $P \in \F_{q^r}[X,\sigma]$ from the factorization of $\Psi(P) \in \F_q[Y]$.
\begin{prop}\label{sqffactorization}
Let $P \in \F_{q^r}[X,\sigma]$ be a monic polynomial. Assume $\Psi(P) = Q_1\cdots Q_s$, with $Q_i \in \F_q[Y]$ distinct irreducible monic polynomials ($\Psi(P)$ is squarefree). Then $P$ has a unique right-divisor $P_i$ such that $\Psi(P_i) = Q_i$. It is given by $P_i = rgcd(P,Q_i(X^r))$. 
\end{prop}
\begin{proof}
Assuming that the $P_i$'s are irreducible, uniqueness is clear, since a right-divisor $S$ of $P$ such that $\Psi(S) = Q_i$ must divide both $P$ and $Q_i(X^r)$. Let $1 \leq i \leq s$, and $P_i = rgcd(P, Q_i(X^r))$. Since $V_{\Psi(P)(X^r)} = V_P\otimes_{\F_q}\F_{q^r}$, $\Psi(\Psi(P)(X^r)) = \Psi(P)(X)^r$. Hence, all divisors of $\Psi(P_i)$ are in the same similarity class, and $\Psi(P_i)$ is a power of $Q_i$, so its degree is divisible by $\deg Q_i$. But $Q_i$ has only multiplicity one as a divisor of $\Psi(P)$, so $\Psi(P_i)$ is either 1 or $Q_i$. Since $\sum_i \deg P_i = \deg P$, $\Psi(P_i) = Q_i$ for all $1 \leq i \leq s$, and $P_i$ has degree $\deg Q_i$ and is irreducible because $Q_i$ is. 
\end{proof}
\begin{rmq}
More generally, when $\Psi(P) = Q_1^{e_1}\cdots Q_s^{e_s}$, $P_i = rgcd(P,Q_i(X^r))$ has degree divisible by $\deg Q_i$, and $\Psi(P_i) = Q_i^{\varepsilon_i}$ for some $1 \leq \varepsilon_i \leq e_i$. This can sometimes provide a partial or even complete factorization for $P$, but not always: this will be better understood later when we count factorizations of a given polynomial.
\end{rmq}
\subsection{Counting irreducible polynomials}
Before counting factorizations of a given skew polynomial, we focus on finding the number of monic irreducible skew polynomials. This computation appears in \cite{thlio}, although it is obtained by very different methods. Here, it only comes from the computation of the cardinal of the fibers of $\Psi$.\\
Recall that a $\varphi$-module is called simple if it has no nontrivial subspaces stable by $\varphi$.
\begin{lem}\label{Endobjsimple}
Let $D$ be a simple étale $\varphi$-module over $\F_{q^r}$ of dimension $d$. Then $\textrm{End}(D) = \F_{q}[\varphi^r] \simeq \F_{q^d}$.
\end{lem}
\begin{proof}
Let $E = \textrm{End}(D)$. It is clear that $\F_{q}[\varphi^r]$ is contained in $E$. Moreover, any $u \in E$ commutes with $\varphi$ and therefore with $\varphi^r$. Since $D$ is simple, $\varphi^r$ has no nontrivial invariant subspace, so it is a result of elementary linear algebra that the commutant of $\varphi^r$ is $\F_{q^r}[\varphi^r]$. Hence $E$ is contained in$\F_{q^r}[\varphi^r]$. Now let $u = \sum_{i=0}^{d-1} a_i \varphi^r \in E$. The condition that $u$ commutes with $\varphi$ yields $\left(\sum_{i=0}^{d-1} (a_i^q - a_i) \varphi^r\right)\varphi = 0$. Hence, the endomorphism $\sum_{i=0}^{d-1} (a_i^q - a_i) \varphi^r$ is zero on the image of $\varphi$. Since $D$ is étale, $\sum_{i=0}^{d-1} (a_i^q - a_i) \varphi^r$ is zero, and since $(\text{id}, \varphi^r, \dots, \varphi^{(d-1)r})$ is a basis of the commutant of $\varphi^r$ over $\F_{q^r}$, $a_i^q = a_i$ for all $0\leq i \leq d-1$, so $u \in \F_{q}[\varphi^r]$. Hence $E$ has dimension $d$ over $\F_q$, and it is isomorphic to $\F_{q^d}$.
\end{proof}
\begin{prop} \label{cardfibres}
Let $Q \in \F_q[Y]$ be a monic irreducible polynomial of degree $d$. Then the number of monic polynomials $P \in \F_{q^r}[X, \sigma]$ such that $\Psi(P) = Q$ is $\frac{q^{dr}-1}{q^d-1}$.
\end{prop}
\begin{proof}
By Corollary \ref{psicaracterise}, it is enough to compute the number of polynomials $P$ similar to a given $P_0$ such that $\Psi(P) = Q$. Let $P$ be such a polynomial. Since $Q$ is irreducible, so is $P$, and therefore the $\varphi$-module $D_P$ is simple. We already know that any polynomial similar to $P$ appears as a semi-characteristic polynomial of $\varphi$.\\
Now we want to characterize the nonzero $x,y \in D_P$ such that $\chi_{\varphi,x} = \chi_{\varphi,y}$. We claim that these are the nonzero $x, y$ such that $y = u(x) $ for some $u \in \textrm{End}(D_P)$. Indeed, it this is the case, then $\chi_{\varphi,x}(\varphi)(y) = u(\chi_{\varphi,x}(\varphi)(x)) = 0$, so $\chi_{\varphi,y} = \chi_{\varphi,x}$. Conversely, if $\chi_{\varphi,x} = \chi_{\varphi,y}$ then the map $x \mapsto y$ defines an automorphism of $\F_{q^r}$-vector space of $D_P$ that is $\varphi$-equivariant thanks to this relation. This shows, using Lemma \ref{Endobjsimple}, that there is a natural bijection between $D_P$ modulo the relation $\chi_{\varphi,x} = \chi_{\varphi,y}$ and $D_P$ modulo the relation of End($D_P$)-colinearity.\\
Putting both parts together, we get the fact that $\{$Monic polynomials similar to $P \}$ is in bijection with $\{ $End($D_P$)-lines in $D_P \}$, yielding the result.
\end{proof}
\begin{cor}
The number of monic irreducible polynomials of degree $d$ in $\F_{q^r}[X, \sigma]$ is 
$$\frac{q^{dr} -1}{d(q^d-1)} \sum_{i\mid d} \mu\left( \frac{i}{d}\right) q^i,$$
where $\mu$ is the Möbius function. 
\end{cor}
\begin{proof}
It follows directly from Corollary \ref{psicaracterise}, Proposition \ref{cardfibres} and the classical formula for the number of irreducible monic polynomials in $\F_q[Y]$, that can be for instance found in \cite{ln}. As mentioned before, this formula already appeared in \cite{thlio}.
\end{proof}
\subsection{A closer look at the structure of $D_P$}
In this section, we consider the whole structure of the $\varphi$-module $D_P$ instead of just looking at $\Psi(P)$. We address two different problems that both need a careful look at the structure of a $\varphi$-module $D$, or equivalently of the associated representation. Note that Proposition \ref{sqffactorization} shows that when $\Psi(P)$ has no square factors, for each choice of an order of the similarity classes of the polynomials arising in a factorization of $P$, there is a unique factorization of $P$ such that $P$ has its factors in the chosen order. Therefore, there are exactly $s!$ factorizations of $P$ as a product of irreducible monic polynomials in this case. The starting point of our discussion is Proposition \ref{diviseursetquotients}. We rewrite it in our context, adding the formulation coming from representation theory.
Let $P \in \F_{q^r}[X,\sigma]$ be a monic polynomial with nonzero constant coefficient. There are natural bijections between the following sets:
\begin{enumerate}[(i)]
\item The monic irreducible right-divisors of $P$;
\item The simple quotients of the $\varphi$-module $D_P$;
\item The irreducible subrepresentations of $V_P$.
\end{enumerate}
%\begin{rmq}
%In this case, there is a natural bijection between the factorizations of $P$ as a product of monic irreducible polynomials, and the Jordan-Hölder sequences of $V_P$. This allows us to give a different computation for Proposition \ref{cardfibres}. If $P \in \F_{q^r}[X, \sigma]$ is irreducible, then the monic irreducible right-divisors of $Q=\Psi(P)(X^r)$ are exactly the monic polynomials that are similar to $P$. Indeed, such a polynomial divides $Q$ by Corollary \ref{psicaracterise}. Conversely, Lemma \ref{isomrep} shows that $V_{\Psi(P)(X^r)}$ is isomorphic to the direct sum of $r$ copies of $V_P$, so all the irreducible right-divisors of $Q$ are similar to $P$. Now, the number of subrepresentations of $V_Q$ that are isomorphic to $V_P$ is computed the following way: any irreducible subrepresentation of $V_Q$ is isomorphic to $V_P$. Moreover, it is determined by a choice of a nonzero $x \in V_Q$ such that $\Psi(P)(g)(x) = 0$. Since $V_Q = V_P^{\oplus r}$, there are $(q^d)^r -1$ choices for $x$, and two different choices $x$, $x'$ give rise to the same subspace if and only if $x$ and $x'$ are $\F_{q^d}$-colinear. Hence, there are exactly $\frac{q^{dr} -1}{ q^d -1}$ monic irreducible divisors of $\Psi(P)(X^r)$, which is also the number of monic polynomials similar to $P$, and the cardinal of a the fiber of $\Psi$ above $\Psi(P)$.
%\end{rmq}
\begin{rmq}\label{nbjordanblocks}
In this context, this result may sound surprising, because if $V_0 = V^{\oplus s}$ with $V$ an irreducible representation of dimension $d$ and $s > r$, $V_0$ has $\frac{q^{ds}-1}{q^d -1}$ distinct subrepresentations with irreducible quotient, whereas all the divisors of $P$ are similar, and hence $P$ has less than $\frac{q^{dr}-1}{q^d -1}$ monic irreducible right-divisors. There is no contradiction, however: the proposition just says that this case never happens, meaning that in this case $V_0$ is not some $V_P$, or, equivalently, that the $\varphi$-module associated  to $V_0$ cannot be generated by a single element if $s \geq r$. We can give yet more precise information about whether there is a generator for a $\varphi$-module over a finite field: because of the equivalence of categories with the representations, the $\varphi$-module is a direct sum of $\varphi$-modules such that the composition factors of each summand are all the same, and the $\varphi$-module has a generator if and only if each of them has one. It remains to decide when a $\varphi$-module with isomorphic composition factors (that is, with semi-simplification isomorphic to a direct sum of copies of the same simple object) has a generator. We will introduce some definitions to discuss this matter. They will also be useful to count factorizations.
\end{rmq}
\subsubsection{Generated $\varphi$-modules over finite fields}\label{generators}
We recall that an endomorphism is in Jordan form if its matrix is block-diagonal, where the blocks have the following form:
$$ \begin{pmatrix} A & I & 0 & \dots &0\\ 0 & A & I & \dots & 0\\ 0 & 0 & \ddots & \ddots & \vdots \\ 0 & \dots & 0 & \ddots & I \\ 0 & \dots & \dots & 0 & A  \end{pmatrix},$$
where the characteristic polynomial of $A$ is irreducible. These blocks are called the Jordan blocks, and the length of a Jordan block is the number of $A$ in this writing (it is the length of the $\F_q[Y]$-module corresponding to the endomorphism whose matrix is the Jordan block). Note that if the characteristic polynomial of the endomorphism is split, the matrices $A$ that appear in its Jordan form are 1-dimensional. Assume that the minimal polynomial $\mu_g$ of $g$ is a power of an irreducible polynomial $\pi$, say $\mu_g = \pi^t$, with $\deg \pi = \delta$. Let $W = \F_q^\delta$ endowed with the endomorphism whose matrix in the canonical basis is the companion matrix of $\pi$ : any irreducible invariant subspace of $g$ is isomorphic to $W$. We say that $g$ has type $(t_1, \dots, t_m)$ if $t_1\geq \dots \geq t_m$ and the Jordan blocks of (the Jordan form of) $g$ have lengths $t_1, \dots, t_m$. Of course, $t = t_1$. In general, if $g$ has type $(t_1, \dots, t_m)$, then the map induced by $g$ on any quotient of $V$ by an irreducible invariant subspace has type $(t_1', \dots, t_m')$ where $t_i' = t_i$ except for one $i$ for which $t_i' = t_i -1$ (or possibly $m' = m-1$ when $t_m = 1$).
\begin{lem}\label{semichar-nongen}
Let $D$ be an étale $\varphi$-module over $\F_{q^r}$. Assume all the composition factors of $D$ are isomorphic. Then $D$ has a generator if and only if the number of Jordan blocks of the representation $V$ associated to $D$ is $\leq r$.
\end{lem}
\begin{proof}
Since all the composition factors of $D$ are isomorphic, it makes sense to talk about the type of the associated representation $V$, on which the Frobenius acts by $g$. The proof is done by induction on the type of $g$. The possible types for the considered endomorphisms have the form $(t_1, \dots, t_s)$ with $s \leq r$ by hypothesis. We complete the notation with zeros in order that the type is denoted by a $r$-tuple $(t_1, \dots, t_s, 0, \dots, 0)$. We order the types with respect to the lexicographical order. If $g$ has type $(1, 0, \dots, 0)$, then $D$ is simple, so it has a generator. Now assume $g$ has type $(t_1, \dots, t_s, 0, \dots, 0)$. If every simple sub-$\varphi$-module of $D$ is a direct factor of $D$, then $g$ has type $(1, \dots, 1, 0, \dots, 0)$, and $D = D_0^{\oplus s}$ where $D_0$ is the only composition factor of $D$. But we know that $D_0^{\oplus s}$ has a generator since it is a quotient of $D_0^{\oplus r}$ which has one. Now, if $D$ has a simple subobject $D_0$ that is not a direct factor, then we have an exact sequence 
$$ 0 \rightarrow D_0 \rightarrow D \rightarrow D' \rightarrow 0$$
that is not split, and where the representation $V'$ associated to $D'$ has a smaller type than $g$. By induction hypothesis, there exists some $x_0 \in D'$ such that $D'$ is generated by $x_0$. Let $x$ be any lift of $x_0$ in $D$, and let $D_x$ be the sub-$\varphi$-module of $D$ generated by $x$. If $D_x \cap D_0 = \{ 0\}$, then $x_0 \mapsto x$ gives a splitting, which is not possible. Hence $D_0 \subset D_x$, and $D_x = D$, so $D$ has a generator.
\end{proof}
\subsubsection{Counting factorizations}
Now let us consider the problem of counting factorizations of a monic polynomial $P \in \F_{q^r}[X,\sigma]$ as a product of monic irreducible polynomials. The problem is reduced to that of computing the number of Jordan-Hölder sequences for an endomorphism $g$ of an $\F_q$-vector space $V$. As before, first assume that the minimal polynomial $\mu_g$ of $g$ is a power of the irreducible polynomial $\pi$, say $\mu_g = \pi^t$, with $\deg \pi = \delta$, and let $W = \F_q^\delta$ endowed with the endomorphism whose matrix in the canonical basis is the companion matrix of $\pi$. There are $\frac{q^{\delta m} -1}{q^\delta -1}$ irreducible invariant subspaces for $g$ because the intersection of such an invariant subspace with a Jordan block must be the only irreducible invariant subspace of this block, or zero, so $\Hom(W,V) = \Hom(W, W^{\oplus m})$. We want know how many quotients of $V$ by an irreducible invariant subspace there are for each given possible type. 
\begin{lem}\label{pathweight}
Let $(t_1, \dots, t_m)$ be the type of $g$. Let $1 \leq i \leq m$ such that $i=m$ or $t_i> t_{i+1}$. Let $i_0$ be the smallest $j$ such that $t_j = t_i$. Then there are $q^{\delta(i-1)} + q^{\delta i} + \cdots + q^{\delta(i_0-1)}$ invariant irreducible subspaces of $V$ such that the quotient has type $(t_1, \dots, t_i -1, t_{i+1}, \dots, t_m)$ (or $(t_1, \dots, t_{m-1})$ if $i=m$ and $t_m = 1$).
\end{lem}
\begin{proof}
Denote by $(e_{1,1},\dots,e_{1,\delta}, e_{2,1},\dots, e_{2,\delta}, \dots)$ a basis of $V$ in which the matrix of $g$ has Jordan form. More precisely, for all $1\leq i \leq m$, and for all $1\leq j \leq t_i$ and $1\leq l \leq \delta$, we have $g(e_{j,l}) = e_{j,l+1}$ if $(j,l)$ is not of the shape $(j,1)$ for some integer $j \geq 2$, or of the shape $(j,\delta)$ for some integer $j \geq 1$, $g(e_{j, 1}) = e_{\delta u,\delta} + e_{\delta u +1, 2}$ if $j \geq 2$, and $g(e_{j,\delta}) = \sum_{l=1}^\delta a_{l}e_{j,l}$, where $\sum_{l=1}^{\delta} a_l Y^{l-1} \in\F_q[Y]$ is an irreducible polynomial that does not depend of $j$ (it is the characteristic polynomial of the induced endomorphism on any irreducible invariant subspace).\\
There are $i_0 - 1$ Jordan blocks of $g$ whose length is greater than the length of the $i$-th block. For $\lambda = (\lambda_{1,1}, \dots, \lambda_{1,\delta},\dots, \lambda_{i_0-1,\delta}) \in \F_q^{\delta(i_0 -1)}$, let $v_\lambda = e_{i_0,1} + \sum_{j=1}^{i_0 -1} \sum_{l = 1}^\delta \lambda_{j,l} e_{j,l}$. Since two such vectors $v_\lambda$, $v_\mu$ are not colinear, they generate distinct invariant subspaces $V_\lambda$, $V_\mu$, which are clearly isomorphic to $W$. Moreover, the quotient $V/V_\lambda$ has the same type as $V/V_{(0)}$ because the map $V \rightarrow V$ that sends $e_{i_0,1}$ to $v_\lambda$ and is the identity outside the invariant subspace generated by $e_{i_0,1}$ is an isomorphism (its matrix is upper triangular). One can build the same way invariant subspaces with quotients of the same type as generated by vectors of the shape $e_{i_0+1,1} + \sum_{j=1}^{i_0 }  \sum_{l = 1}^\delta \lambda_{j,l} e_{j,l} , \dots, e_{i,1} + \sum_{j=1}^{i -1}  \sum_{l = 1}^\delta \lambda_{j,l} e_{j,l}$. There are exactly $q^{\delta i_0-1} + \cdots + q^{\delta i-1}$ invariant subspaces that are built in this way. Doing such constructions for each $i'$ satisfying the hypotheses of the lemma, we get exactly $\frac{q^{\delta m} -1}{q^\delta -1}$ irreducible invariant subspaces, which means all of them. Among these subspaces, the ones for which the quotient has the requested shape are exactly the  $q^{\delta i_0-1} + \cdots + q^{\delta i-1}$ built for the first $i$ we considered. This proves the lemma.
\end{proof}
In order to compute the number of Jordan-Hölder sequences of $g$, consider the following diagram:
$$\begin{tabular}{|c|c|c|c|}
\multicolumn{1}{c}{1} & \multicolumn{1}{c}{$q^\delta$} & \multicolumn{1}{c}{$\dots$} & \multicolumn{1}{c}{$q^{\delta(m-1)}$}\\
\hline
$t_1$ & $t_2$ & $\dots$ & $t_{m}$\\
\hline
\end{tabular}$$
with $t_1\geq \dots \geq t_m$. An \emph{admissible path} is a transformation of this table into another table  $\begin{tabular}[b]{|c|c|c|c|}
\multicolumn{1}{c}{1} & \multicolumn{1}{c}{$q^\delta$} & \multicolumn{1}{c}{$\dots$} & \multicolumn{1}{c}{$q^{\delta(m'-1)}$}\\
\hline
$t_1'$ & $t_2'$ & $\dots$ & $t_{m'}'$\\
\hline
\end{tabular}$ such that 
\begin{itemize}
\item either $m' = m-1$, $t_j' = t_j$ for $1 \leq j \leq m-1$, if $t_m = 1$;
\item or $m' = m$, $t_j' = t_j$ for all $j\neq i$, with $1 \leq i \leq m$ such that $t_{i} > t_{i+1}$.
\end{itemize}
To such a path $\gamma$, we affect a weight $w(\gamma)$, which is the sum of the coefficients written above the cells of the first table containing the same number $t_i$ as the cell whose coefficient was lowered in the second table. Here is an example of a table and all the admissible paths with the corresponding weights:
$$
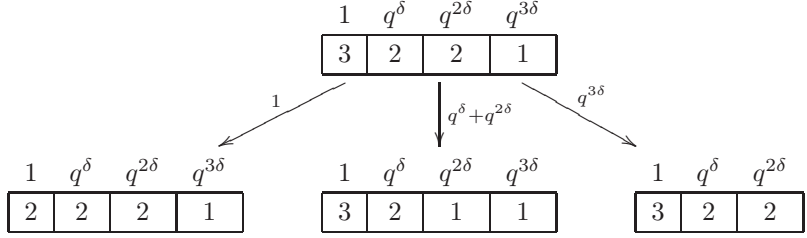
\xymatrix{
& {\begin{tabular}{|c|c|c|c|}
\multicolumn{1}{c}{1} & \multicolumn{1}{c}{$q^\delta$} & \multicolumn{1}{c}{$q^{2\delta}$} & \multicolumn{1}{c}{$q^{3\delta}$}\\
\hline
$3$ & $2$ & $2$ & $1$\\
\hline
\end{tabular}} \ar[dl]_{1} \ar[d]^{q^\delta + q^{2\delta}} \ar[dr]^{q^{3\delta}}\\
{\begin{tabular}{|c|c|c|c|}
\multicolumn{1}{c}{1} & \multicolumn{1}{c}{$q^\delta$} & \multicolumn{1}{c}{$q^{2\delta}$} & \multicolumn{1}{c}{$q^{3\delta}$}\\
\hline
$2$ & $2$ & $2$ & $1$\\
\hline
\end{tabular}}
&
{\begin{tabular}{|c|c|c|c|}
\multicolumn{1}{c}{1} & \multicolumn{1}{c}{$q^\delta$} & \multicolumn{1}{c}{$q^{2\delta}$} & \multicolumn{1}{c}{$q^{3\delta}$}\\
\hline
$3$ & $2$ & $1$ & $1$\\
\hline
\end{tabular}}
&
{\begin{tabular}{|c|c|c|}
\multicolumn{1}{c}{1} & \multicolumn{1}{c}{$q^\delta$} & \multicolumn{1}{c}{$q^{2\delta}$} \\
\hline
$3$ & $2$ & $2$\\
\hline
\end{tabular}}
}$$
By Lemma \ref{pathweight}, the weight of an admissible path from one table to another, is the number of irreducible invariant subspaces of an endomorphism $g$ with type given by the first table such that the quotient has the type given by the second table. Therefore, a sequence of admissible paths ending to an empty table represents a class of Jordan-Hölder sequences. Thus the number of distinct sequences in this class is the product of the weights of the paths along the sequence. Hence, the number of Jordan-Hölder sequences for $g$ is $\sum_{(\gamma_1,\dots \gamma_\tau)} \prod_{i=1}^{\tau} w(\gamma_i)$ the sum being taken on all sequences $(\gamma_1, \dots, \gamma_\tau)$ of admissible paths ending at the empty table (so $\tau = \sum_{j=1}^m t_j$).\\

In general, we do not know any simple formula for this computation (except in some particular cases that we shall discuss below), but a recursive algorithm can be used to compute the result. Given a table with coefficients $(t_1, \dots t_m)$, we need to compute the values associated to all tables that can be built out of this table through an admissible path. There are at most $t_1\dots t_m$ such tables. According to Remark \ref{nbjordanblocks}, there are at most $r$ Jordan blocks for $g$. Using the notations above, this means $m \leq r$. Moreover, $\sum_{i=1}^m \delta t_i = \dim V$, so that $\tau \leq \frac{\dim V}{\delta}$. Then, by the arithmetic-geometric inequality, $t_1\dots t_m \leq \left(\frac{\dim V}{\delta}\right)^r$, so the computation of the number of Jordan-Hölder sequences of $g$ can be done in polynomial time in the dimension of $V$, when $r$ is fixed.
\begin{ex}\label{ex1}
Let us take a closer look at one particular example: assume that the type of $g$ is $(1, \dots, 1)$ ($m$ terms). Then there is only one admissible path $\gamma$, that leads to $(1, \dots, 1)$ ($m-1$ terms), and its weight is $\frac{q^{m\delta} -1}{q^\delta-1}$. Hence the number of Jordan-Hölder sequences of $g$ is $\prod_{j=1}^m \frac{q^{\delta j} -1}{q^\delta -1} = [m]_{q^\delta} !$, the $q^\delta$-factorial of $m$.
\end{ex}
\begin{ex}
Let us look at an actual example. Let $q = 7$, $r=2$, with $\F_{7^2}$ defined as $\F_7[Y]/(Y^2 -Y +3)$, and let $\omega$ be the class of $Y$ in $\F_{7^2}$. Consider the polynomial $P = X^6 + \omega^{3}X^5 + \omega^{17}X^4 + \omega^{3}X^3 + \omega^{27} X^2 +\omega^{35}X + \omega^{36}$. The Jordan form of the matrix of $\varphi^2$ on $D_P$ is $\begin{pmatrix} 0&-4&0&0&0&0\\ 1 & -1 &0 &0 &0&0 \\  0&0 & 0& -4 &0&0\\ 0 &0&1 & -1 & 1&0\\ 0&0&0&0 &0 & -4 \\ 0 & 0 &0 &0 & 1 & -1 \end{pmatrix}$, so the endomorphism $g$ has type $(2,1)$ with 2-dimensional irreducible blocks. We write the following diagram with all the admissible paths and their weights:
$$
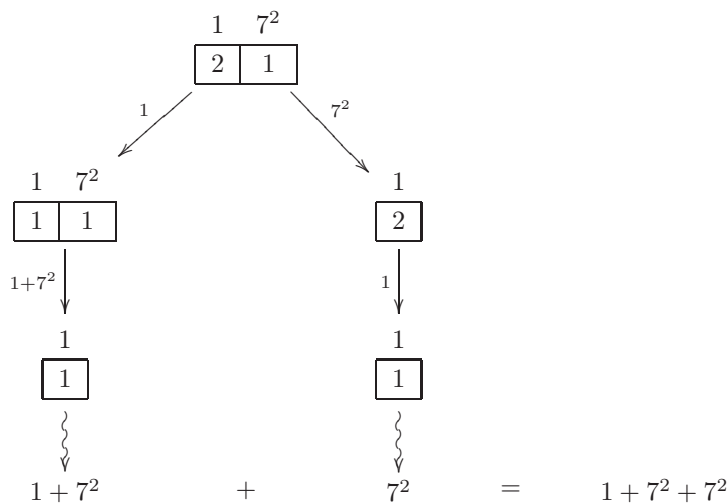
\xymatrix{
& {\begin{tabular}{|c|c|}
\multicolumn{1}{c}{1} & \multicolumn{1}{c}{$7^2$} \\
\hline
$2$ & $1$\\
\hline
\end{tabular}} \ar[dl]_{1} \ar[dr]^{7^{2}}\\
{\begin{tabular}{|c|c|}
\multicolumn{1}{c}{1} & \multicolumn{1}{c}{$7^2$} \\
\hline
$1$ & $1$\\
\hline
\end{tabular}} \ar[d]_{1+7^2}
&
&
{\begin{tabular}{|c|}
\multicolumn{1}{c}{1} \\
\hline
$2$\\
\hline
\end{tabular}} \ar[d]_{1}\\
{\begin{tabular}{|c|}
\multicolumn{1}{c}{1} \\
\hline
$1$\\
\hline
\end{tabular}} \ar@{~>}[d]
&
&
{\begin{tabular}{|c|}
\multicolumn{1}{c}{1} \\
\hline
$1$\\
\hline
\end{tabular}} \ar@{~>}[d] \\
1 + 7^2 & + & 7^2 & = & 1+ 7^2 + 7^2
}$$ 
This shows that the number of factorizations of $P$ as a product of monic irreducible polynomials is 99. An exhaustive research of all the factorizations with \textsc{Magma} gives the same result, but takes around one minute, whereas this computation is instantaneous.
\end{ex}
Now, we need to look at the general case, with no further assumption on the minimal polynomial of $g$. In this case, by the Chinese remainders Theorem, $V$ is a direct sum of invariant subspaces on which the induced endomorphisms have minimal polynomial that is a power of an irreducible. Here, the type of $g$ is defined again as the data of $((W_1,T_1), \dots, (W_s,T_s))$ where the $W_l$'s are the distinct classes of irreducible invariant subspaces of $V$, and the $T_l$'s are the tables representing the types of the endomorphisms induced on the corresponding subspaces of $V$. The notion of admissible path can be defined as previously.
\begin{prop}
Let $g$ be an endomorphism of an $\F_q$-vector space $V$. Assume that the type of $g$ is $((W_1,T_1), \dots, (W_s,T_s))$. Denote by $\delta_i$ the dimension of $W_i$, and by $ \tau_i$ the sum of the coefficients in table $T_i$. Then the number of Jordan-Hölder sequences of $g$ is 
$$\frac{(\tau_1 + \cdots + \tau_s)!}{\tau_1!\cdots \tau_s!} \prod_{(\Gamma_1, \dots, \Gamma_s)} w(\Gamma_1)\cdots w(\Gamma_s),$$
the product being taken over all the $s$-uples $(\Gamma_1, \dots, \Gamma_s)$ of admissible path sequences ending at the empty tables.
\end{prop}
\begin{proof}
From a chain of admissible paths ending at $((W_1, \emptyset), \dots, (W_s, \emptyset))$, it is possible to extract its $W_l$-part $\Gamma_l$ for all $1 \leq l \leq s$. By definition, it is the sequence of all the paths involving a change in the table associated to $W_l$. Such a chain is a sequence of admissible paths from $T_l$ ending at the empty table. It is clear that the weight of the path sequence is the product of the weights of the $\Gamma_l$'s. Therefore, it does not depend on the way the $\Gamma_l$'s were combined together. The admissible path sequences that end at $((W_1, \emptyset), \dots, (W_s, \emptyset))$ are all the different ways to recombine admissible path sequences from all the $(W_i,T_i)$ to the empty table. The weight of such a sequence is the product of the weights of the $W_l$-parts. There are as many recombinations as anagrams of a word that includes $\tau_l$ times the letter $W_l$ for all $1 \leq l \leq s$, $\tau_l$ being the sum of the integers appearing in $T_l$. The result then follows directly from the previous discussion an the fact that the number of anagrams of a word that includes $\tau_l$ times the letter $W_l$ is the multinomial coefficient $\frac{(\tau_1 + \cdots + \tau_s)!}{\tau_1!\dots \tau_s!} $.
\end{proof}
\begin{ex}
Assume $g$ has type $((W_1, (t_1)) , \dots, (W_s, (t_s)))$. It is easy to see that the only admissible path sequence for $(W_l,(t_l))$ has weight 1. Hence the number of Jordan-Hölder sequences of $g$ is $\frac{(t_1 + \cdots + t_s)!}{t_1!\dots t_s!}$.
\end{ex}
The previous discussions also allow us to explain how to find all factorizations of a given polynomial $P$ using Giesbrecht's algorithm. A first factorization of $P$ yields a Jordan-Hölder sequence for the $\varphi$-module $D_P$. All the simple sub-$\varphi$-modules of $D_P$ can be constructed as in the proof of Lemma \ref{pathweight}. Any such simple sub-$\varphi$-module yields a factorization of $P$ as $P = P_1Q$, with $P_1$ irreducible as in Theorem \ref{submodule}. By performing left-euclidean division in $\F_{q^r}[X, \sigma]$ (which is possible because $\F_{q^r}$ is perfect), we can find $Q$, which we factor again by Giesbrecht's algorithm. For each factorization we find, it takes as many uses of Giesbrecht's algorithm as factors there are in the polynomial. Since Giesbrecht's algorithm is polynomial in the degree and in $r$, this quite naïve method gives all the factorizations of $P$ with a complexity that is a polynomial in $d$ and $r$ times the number of factorizations of $P$.
\begin{algorithm}
\caption{AllFactorizations($P$) returns all the factorizations of $P \in \F_{q^r}[X, \sigma]$}
\begin{algorithmic}
\REQUIRE{$P \in  \F_{q^r}[X, \sigma]$, monic}
\ENSURE{List of all possible factorizations of $P$ as product of monic irreducibles}
\STATE{Compute a factorization of $P$, $P = P_1\dots P_s$}
\IF{$s = 1$}
\RETURN{$[P]$}
\ELSE
\STATE{ Let $G$ be the companion matrix of $P$, compute a Jordan-Hölder sequence for the $\varphi$-module that has matrix $G$ as in Proposition \ref{submodule}}
\STATE{Let $\mathcal{F}_P = [~]$}
\FOR{ each isomorphism class $C$ of submodules of $D_P$}
\FOR{ $x \in D_P$ such that the submodule generated by $x$ is in $C$}
\STATE {Compute $\mathcal{F}$ = AllFactorizations($\chi_{\varphi,x}^{-1}P$)}
\STATE {Add $(\chi_{\varphi,x}, F)$ to $\mathcal{F}_P$, for all $F \in \mathcal{F}$}
\ENDFOR
\ENDFOR
\RETURN{$\mathcal{F}_P$}
\ENDIF
\end{algorithmic}
\end{algorithm}

Note that the same kind of methods could also be used to find all the factorizations of a polynomial with prescribed orders of the similarity classes of the factors appearing in $P$, or of just some of them.
\end{spacing}

\end{document}